\documentclass[reqno]{amsart}
\usepackage{cancel}
\usepackage{stmaryrd}
\usepackage{textcomp}
\usepackage{tkz-euclide}
\usepackage{hyperref}
\usepackage{amsmath}
\usepackage{amssymb} 
\usepackage{mathrsfs}
\usepackage{graphicx}
\usepackage{tikz}
\usepackage[font=footnotesize, labelfont=bf]{caption}
\usepackage{subcaption}
\usepackage{xcolor}  
\usepackage{amsthm}
\usepackage{float}
\usepackage{enumitem}
\usepackage[english]{babel}
\newcommand{\res}{\mathop{\hbox{\vrule height 7pt width 0.5pt depth 0pt
			\vrule height 0.5pt width 6pt depth 0pt}}\nolimits}

\definecolor{LimeGreen}{cmyk}{0.50, 0.5, 1, 0}

\newcommand{\eps}{\varepsilon} 
\newcommand{\dx}{\, {\rm d}x}
\newcommand{\dy}{\, {\rm d}y}

\newcommand{\ds}{\, {\rm d}s}
\newcommand{\dt}{\, {\rm d}t}

\newcommand{\e}{\varepsilon}

\newcommand{\Sph}{{\mathbb S}}

\theoremstyle{plain}
\begingroup
\newtheorem{theorem}{Theorem}[section]
\newtheorem{lemma}[theorem]{Lemma}
\newtheorem{proposition}[theorem]{Proposition}

\endgroup

\numberwithin{equation}{section}

\newcommand{\N}{\mathbb{N}}

\newcommand{\R}{\mathbb{R}}

\renewcommand{\S}{\mathcal{S}}
\renewcommand{\L}{\mathcal{L}}
\renewcommand{\H}{\mathcal{H}}
\newcommand{\M}{\mathcal{M}}
\newcommand{\dHn}{\, {\rm d}\H^{n-1}}

\newcommand{\loc}{\mathrm{loc}}

\newcommand{\defas}{:=}
\newcommand{\wto}{\rightharpoonup}
\newcommand{\wsto}{\overset{*}{\wto}}

\newcommand{\om}{\omega}
\newcommand{\Per}{\mathrm{Per}\,}

\newcommand{\sfV}{\mathsf{V}}
\newcommand{\sfIV}{\mathsf{IV}}
\newcommand{\sfRV}{\mathsf{RV}}

\theoremstyle{definition}
\begingroup
\newtheorem{definition}[theorem]{Definition}

\newtheorem{notation}[theorem]{Notation}
\endgroup
\theoremstyle{remark}
\newtheorem{remark}[theorem]{Remark}
\renewcommand{\tilde}{\widetilde}

\renewcommand{\d}{\, \mathrm{d}}
\usepackage[normalem]{ulem}

\newcommand{\supp}{\mathrm{supp}\,}
\newcommand{\Hu}{\hat H}
\begin{document}

\title[Phase-transition functionals]
{Interacting phase fields yielding phase separation on surfaces}
\author[B. Lledos]{Benjamin Lledos}
\address[]{Université catholique de Louvain, Belgium}
\email[Benjamin Lledos]{benjamin.lledos@uclouvain.be}
\author[R. Marziani]{Roberta Marziani}
\address[]{Università degli studi dell'Aquila, Italy}
\email[Roberta Marziani]{roberta.marziani1@univaq.it}
\author[H. Olbermann]{Heiner Olbermann}
\address[]{Université catholique de Louvain, Belgium}
\email[Heiner Olbermann]{heiner.olbermann@uclouvain.be}


\maketitle

{\small
\noindent 

\medskip


\begin{abstract}
  In the present article we study   diffuse interface models  for two-phase biomembranes.  We will do so by starting off with a diffuse interface model on $\R^n$ defined by two coupled phase fields $u,v$. The first phase field $u$ is the diffuse  approximation of the  interior of the membrane; the second phase field $v$ is  the diffuse approximation of the two phases of the membrane. We prove a compactness result and a lower bound in the sense of $\Gamma$-convergence for pairs of phase functions $(u_\e,v_\e)$. As an application of this first result, we consider a diffuse approximation of a two-phase Willmore functional plus line tension energy. 
\end{abstract}

\section{Introduction}

The rigorous variational framework for the diffuse  approximation of the area of a hypersurface in $\R^n$   is the well-known analysis by Modica and Mortola \cite{MoMo77,Modi87}: To a smooth function $u$ on $\Omega\subseteq \R^n$, one associates the Van der Waals-Cahn-Hilliard energy
  \begin{equation}\label{eq:0}
M_\e(u)=\int_\Omega \left(\frac{\e}{2}|\nabla u|^2+\frac{W(u)}{\e}\right)\dx\,.
\end{equation}
Here $W$ is a non-negative double-well potential with exactly two zeros located at $u=\pm1$, for example $W(u)=(1-u^2)^2$. The integrand above will be called ``Modica-Mortola integrand'' in the sequel. 
In the sharp interface limit $\e\to 0$, $M_\e$ converges in the sense of $\Gamma$-convergence to a multiple of the perimeter functional 
\[
\mathcal P(u)=
\begin{cases}
  \mathrm{Per}\,E &\text{ if }u=2\chi_E-1 \text{ for some set of finite perimeter }E\\
+\infty & \text{ else. }
\end{cases}
\]
 In the last decades this result has been generalized in  several directions. For example we mention \cite{Baldo} for the multi-phase case, \cite{Bouchitte} where 
 heterogeneous fluids which may undergo temperature changes are taken into account,  \cite{BaFo, OwSte} for anisotropic models, and \cite{AnBraCP,BraZep,CFG,CFHP,Morfe,Marziani} where 
 the interaction between singular perturbations and homogenization is considered (see also \cite{BMZ,BEMZ1,BEMZ2} for the homogenization of Ambrosio-Tortorelli functionals).

With certain applications in biophysics in mind, we are going to call boundaries of sets of finite perimeter  \emph{membranes} in the sequel. In some of these applications, one is interested in models for membranes that themselves possess two different phases (with different physical properties), and the interface between such phases should again be associated to an  energy measuring its length. We will call such membranes \emph{two-phase membranes}. A concrete  example are the J\"ulicher-Lipowsky energies \cite{JuLi96} associated to a membrane $\partial E\subseteq \R^3$,
\[
  \mathcal{E}(S_1,S_2) = \sum_{j=1,2}\int_{\partial E} \Big(k_1^j (H-H_0^{j})^2+ k_2^j K\Big)\,{\rm d}\mathcal{H}^{2} + \sigma\int_{\Gamma} 1\,{\rm d}{\mathcal H}^1,
\]
where the membrane $\partial E$ is decomposed into the two phases $S_1,S_2$, with common boundary $\Gamma$, $H$ denotes the mean curvature of $\partial E$,  $H_0^{j}$, $j=1,2$ are different reference values for the mean curvature, $K$ denotes Gauss curvature of $\partial E$,  $k_i^j$, $i,j=1,2$, are phase dependent elastic moduli, and $\sigma$ is the interfacial energy density. 

\medskip

 Similar to the manner in which we approximated the interior of the membrane $\partial E$ by a phase function $u$, we may now wish to approximate the two phases $S_1,S_2$ by a phase function $v$, defined on $\partial E$. 

\medskip

Such an effort has been undertaken recently in \cite{OR}, where diffuse interface energies of Modica-Mortola type have been considered on (generalized) hypersurfaces. Using the concept of generalized $BV$ functions on currents from \cite{ADS}, the paper \cite{OR} contains a compactness result and lower bound in the sense of $\Gamma$-convergence for the Modica-Mortola functional evaluated on  sequences of current-function pairs $(S_\e,v_\e)$. As an application of that first result, it also contains  a compactness result and a lower bound in the sense of $\Gamma$-convergence  for the diffuse approximation of a two-phase Willmore functional combined with a line tension energy (for the case $n=3$). Other noteworthy research efforts concerning the variational analysis of multiphase membranes are \cite{ChMV13,Helm13,Helm15} for the rotationally symmetric case, and the analysis of a more general setting without symmetry assumptions in \cite{BrLS20}.

In the present article, we study the diffuse approximation of a two-phase membrane starting from a pair of phase functions $(u,v)$. The diffuse surface energy of the membrane is \eqref{eq:0}, while the diffuse interfacial energy between the two phases of the membrane is given by the integral of the \emph{product} of the Modica-Mortola integrands for $u$ and $v$, 
  \begin{equation}\label{eq:1}
I_\e(u,v)=\int_\Omega \left(\frac{\e}{2}|\nabla u|^2+\frac{W(u)}{\e}\right)\left(\frac{\e}{2}|\nabla v|^2+\frac{W(v)}{\e}\right)\dx\,.
\end{equation}


In the first main contribution of the present paper, we will consider   a sequence $u_\e$ of phase fields such  that $\dfrac{1+u_\e}{2}$ converges strictly in $BV$ towards the indicator function $\frac{1+u}{2}=\chi_E$ of a set of finite perimeter $E$ in $\R^n$.   We assume that  the diffuse  energies $M_\e(u_\e)+I_\e(u_\e,v_\e)$ are uniformly bounded, and show that there exists a  phase function $v$ that has values in $\{-1,1\}$ $\mathcal H^{n-1}$ almost everywhere on the reduced boundary of $E$, and a subsequence of $(u_{\e},v_{\e})_{\e}$ such that this sequence of pairs converges in a suitable  sense to the pair $(u,v)$.

Our second contribution is a translation of the compactness result  and the lower bound in the sense of $\Gamma$-convergence  for the two-phase Willmore functional plus line tension from \cite{OR} to the setting of pairs of coupled phase functions $(u,v)$. 

\medskip

To explain this second result of the present paper in slightly more detail,  let us consider the diffuse approximation of the Willmore energy for the phase function $u$ in $\R^3$,
\begin{equation*}
\mathcal F_\eps(u):= \int_{\R^3} \frac{1}{\eps
}\left(\frac{W'(u)}{\eps}-\eps\Delta u\right)^2\dx\,.
\end{equation*}
It has been shown in \cite{BP,RS} that  $M_\e+\mathcal F_\e$ converges to the sum of the area functional and the Willmore functional as $\eps\to 0$, in the sense of $\Gamma$-convergence. We will consider couplings of the phase field $v$ the diffuse Willmore energy density,
\[
J_\e(u,v)=\int_{\R^3}\frac{a(v)}{\eps
}\left(\frac{W'(u)}{\eps}-\eps\Delta u\right)^2\dx
\]
where $a(v)$ can be thought of as a phase-dependent  elastic modulus. Adopting suitable variants of the hypotheses from \cite{OR}, we make the  assumption that $J_\e$ is uniformly bounded  in a way that allows for an application of the Li-Yau inequality \cite{li1982new}, and which hence guarantees strict convergence in $BV$ of a subsequence of $\frac{1+u_\e}{2}$ to the indicator function $\chi_E$ of a set of finite perimeter $E\subset\R^3$.  We show that there exists a further subsequence such that the pairs $(u_\e,v_\e)$ converge to a limit pair $(u,v)$ as in our first result, and prove that the associated energy functionals (which in the limit is the two-phase Willmore functional with line tension) satisfy a lower bound inequality in the sense of $\Gamma$-convergence in the limit $\e\to 0$. 

\medskip

In the proof of our results we make extensive use, on the one hand, of the analysis of Modica-Mortola type functionals on current-function pairs from \cite{OR}, which in turn is based on the theory of $BV$ functions on currents from \cite{ADS}, and on the other hand, of the analysis of the diffuse Willmore functional from \cite{RS}. Our first main result is obtained by a slicing argument: To each of the slices we may apply the compactness theorem and the lower bound from \cite{OR}. It then remains to show that the limit is the same for each of the slices, which is achieved by a blow-up argument (see Step 2 in the proof of Proposition \ref{prop:compactness}). In the application of the first result to the diffuse approximation of a two-phase Willmore functional, we rely crucially on the estimates for the ``discrepancy measures''
\[
\xi_\e=\left(\frac{\e}{2}|\nabla u_\e|^2-\frac{W(u_\e)}{\e}\right)\L^3
\]
from \cite{RS}. It is shown there that $|\xi_\e|\to 0$. We may use this result to control the behavior of the \emph{push-forward} of measures 
\[
\mu_\eps=\left(\frac{\e}{2}|\nabla u_\e|^2+\frac{W(u_\eps)}{\e}\right)\L^3
\]
under the graph map $x\mapsto (x,u_\e(x))$, and show that this yields sufficiently strong convergence of the 3-tuples  $(\mu_\e,u_\e,v_\e)$ (see Lemma \ref{lem:strong_uv_convergence}) in order to obtain our result from standard lower semicontinuity theorems for measure-function pairs. 

\medskip

In order to obtain full $\Gamma$-convergence statements, we would have to supply upper bound constructions matching the lower bounds provided by Theorems \ref{thm1} and \ref{thm2}. However, we refrain from treating the general construction in the present paper. We only construct upper bounds (corresponding to the situation encountered in Theorem \ref{thm2}) for the case of limits given by a smooth two-dimensional surface in $\R^3$ possessing a smooth one-dimensional interface between the two phases defined on it. For this case, the generalization of the constructions from \cite{BP} is relatively straightforward, and is carried out in Appendix \ref{AppendixA}.  The question how to generalize this construction to non-smooth limit surfaces is left open for future research.

\medskip

The plan of the paper is as follows: In Section \ref{sec:preliminaries} we introduce some notation and preliminaries concerning $BV$ functions, Sobolev spaces and BV spaces on rectifiable currents, oriented varifolds and  measure-function pairs.  In Section \ref{sec:main_results} we state 
and prove our first  result, Theorem \ref{thm1}. 
Our second result, Theorem \ref{thm2}, will be stated and proved in Section \ref{sec:proofthm2}. 
The upper bound construction for the smooth case is contained in Appendix \ref{AppendixA}, whereas  Appendix \ref{AppendixB} contains some properties of Sobolev spaces with respect to measures.

\section*{Acknowledgments}
The authors would like to thank Matthias R\"oger for helpful discussions, in particular for pointing out the reference \cite{roger2008allen}. RM has been funded by the European Union - NextGeneration EU under the Italian Ministry of University and Research (MUR) National Centre for HPC, Big Data and Quantum Computing CN\_00000013 - CUP: E13C22001000006.

\section{Notation and Preliminary results}
\label{sec:preliminaries}

\noindent In this section we collect some notation and recall some results that will be useful throughout the paper.

\subsection{Notation}\label{subs:notation} 

\begin{enumerate}[label=(\alph*)]
\item $n,m\geq 2$ are fixed positive integers; 
\item $\mathcal L^n$ and $\mathcal H^{n-1}$ denote the Lebesgue measure and the $(n-1)$-dimensional Hausdorff measure on $\R^n$, respectively;
\item for every $A\subset\R^n$ let $\chi_A$ denote the characteristic function of the set $A$;
\item $\Omega$ is an open subset of $\R^n$;
\item $\mathcal{M}(\Omega)$ is the space of finite Radon measures on $\Omega$;
\item $BV(\Omega)$ is the space of functions with bounded variation on $\Omega$ (see Section \ref{BV});
\item $B_r(x)$ denotes the open ball in $\R^n$ of radius $r>0$ centered at $x$;
\item $\Lambda_k(\R^n)$ and $\Lambda^k(\R^n)$, $0\le k\le n$, are the space of $k$-vectors and $k$-covectors, respectively, in $\R^n$;
\item $\mathcal D^k(\Omega)$, $0\le k\le n$, denotes the space of all infinitely differentiable $k$-differential forms $\Omega\to\Lambda^k(\R^n)$ with compact support in $\Omega$, and $\mathcal{D}_k(\Omega)$, $0\le k\le n$, is the space of $k$-\textit{currents} on $\Omega$ (see section \ref{currents}).
\end{enumerate}

\medskip

\subsection{Functions of bounded variation}\label{BV}

We say that a map $u\in L^1_{\rm loc.}(\Omega)$ is a function of bounded variation if
\begin{equation}\nonumber
    |Du|(\Omega):=\sup\biggl\{\int_{\Omega}u\,
    \textrm{div}\phi\dx :\ \phi\in {C}^1_c(\Omega;\R^{n}),\ \|\phi\|_{L^\infty(\Omega)}\leq1\biggr\}<\infty\,.
\end{equation}
We denote by $BV(\Omega)$ the set of such maps. In this case, $Du$ is a vector-valued finite Radon measure and $|Du|$ is the \emph{total variation} of $u$. If $u\in W^{1,1}(\Omega)$, then $Du$ is just the weak gradient of $u$.

We say that a set $E$ has \emph{finite perimeter} in $\Omega$ if $\chi_E\in BV(\Omega)$ and we write $\Per(E,\Omega):=|D\chi_E|(\Omega)$. When $\Omega=\R^n$ we simply write $\Per(E)$.

\begin{definition}\label{StrictBVConc}
    We say that a sequence $(u_\eps)_{\eps>0}$ converges strictly in $BV_{\rm loc.}(\R^n)$ to $u$ if
 $u_\eps$ converges to $u$ in $L_{\rm loc.}^1(\R^n)$ and $|Du_\eps|(\R^n)$ converges to $|Du|(\R^n)$.
\end{definition}

\begin{definition}[Measure theoretic boundary]
    Let $E$ be a set of finite perimeter and $x\in\R^n$. We say $x\in\partial_*E$, the measure theoretic boundary of $E$ if 
    \begin{equation*}
        \limsup_{r\to 0} \frac{\L^n(B_r(x)\cap E)}{r^n}>0
    \end{equation*}
    and
    \begin{equation*}
         \limsup_{r\to 0} \frac{\L^n(B_r(x)\backslash E)}{r^n}>0.
    \end{equation*}
\end{definition}
\begin{proposition}
  Let $E$ be a set of finite perimeter. For $\mathcal{H}^{n-1}$ a.e. $x\in \partial_*E$, the generalized normal to $E$:
 \begin{equation*}
		 \nu_E(x):=\lim\limits_{r\rightarrow0}\dfrac{\int_{B(x,r)} D\chi_E}{\int_{B(x,r)} |D  \chi_E|}
		\end{equation*}
exists and $|\nu_E(x)|=1$.
\end{proposition}

The family of sets with finite perimeter can be identified with the space $BV(\Omega;\{0,1\})$, that is, the space of functions in $BV(\Omega)$ taking values in $\{0,1\}$ almost everywhere. Indeed if $u\in BV(\Omega;\{0,1\})$ then $u=\chi_E$ with $E=\{x\in\Omega\colon u(x)=1\}$ and 
\begin{equation*}
    Du(B)=\int_{B\cap J_u}\nu_u\dHn\,,
\end{equation*}
for every Borel set $B\subset\R^n$, where $J_u$ is the set of approximate jump points of $u$ which, up to $\mathcal{H}^{n-1}$-negligible sets, coincides with $\partial_*E\cap\Omega$ and $\nu_u$ is the external normal to $J_u$ which coincides with $\nu_E$ $\mathcal{H}^{n-1}$-a.e. Moreover 
\begin{equation*}
    \Per(E,\Omega)=|Du|(\Omega)=\mathcal H^{n-1}(J_u\cap\Omega)\,.
\end{equation*}
 In a similar manner any $u\in BV(\Omega;\{-1,1\})$ is of the form $u=2\chi_E-1$ for some  set of finite perimeter $E$.
\begin{proposition}[Co-area formula for $BV$-functions]\label{CoAreaBV}
    If $u\in BV(\Omega)$, then for every non-negative measurable function $g$ we have that
    \begin{equation}\nonumber
        \int_\Omega g\d|Du|=\int_\R\int_{\partial_*\{u\geq t\}} g \dHn\dt\,.
    \end{equation}
\end{proposition}

\subsection{Currents and BV functions on currents}\label{currents}
For $0\le k\le n$, we denote by $\Lambda_k(\R^n)$ the space of all $k-$vectors and by $\Lambda^k(\R^n)$ the space of all $k-$covectors. 
The Hodge star isomorphism is denoted by $$\star:\Lambda_k(\R^n)\to \Lambda_{n-k}(\R^n)\,. $$
For $\Omega\subset\R^n$ an open set, we denote by $\mathcal{D}^k(\Omega)$ the space of all infinitely differentiable $k-$differential forms $\Omega\to \Lambda^k(\R^n) $ with compact support in $\Omega$.
The dual $\mathcal{D}_k(\Omega)$ of $\mathcal{D}^k(\Omega)$ is the space of $k-$currents on $\Omega$. The boundary $\partial T\in \mathcal{D}_{k-1}(\Omega)$ of a current $T   \in \mathcal{D}_k(\Omega)$ is defined as:
\begin{equation*}
    \langle\partial T,\omega\rangle = \langle T,\d\omega\rangle \text{ for all } \omega\in \mathcal{D}^{k-1}(\Omega).
\end{equation*}

\begin{definition}
    The mass of a current $T\in \mathcal{D}_k(\Omega) $ is
    \begin{equation*}
        M_\Omega(T)=\sup\{ \langle T,\omega\rangle\,: \omega\in \mathcal{D}^k(\Omega)\,, \|\omega\|_{L^\infty}\leq 1\}\,.
    \end{equation*}
     Moreover we say that a sequence $(T_h)_{h\in\N}\in \mathcal{D}_k(\Omega)$ converges to $T\in \mathcal{D}_k(\Omega)$ in the sense of currents, and we write $T_h\stackrel{*}{\rightharpoonup}T$, if
    \begin{equation*}
        \langle T_h,\omega\rangle\to\langle T,\omega\rangle\,, \quad\text{ for all }\omega\in \mathcal{D}^k(\Omega)\,.
    \end{equation*}
\end{definition}
If $M_\Omega(T)<+\infty$, then by the Riesz representation theorem there exists a Radon measure $\|T\|$ on $\Omega$ and a $\|T\|$-measurable function $\overrightarrow{T}:\Omega\to \Lambda_k(\R^n)$ such that $|\overrightarrow{T}|=1$ $\|T\|$ a.e. on $\Omega$ and
\begin{equation*}
    \langle T,\omega\rangle=\int_{\Omega}\langle\omega,\overrightarrow{T}\rangle\d \|T\|\text{ for all }\omega\in \mathcal{D}^k(\Omega)\,.
\end{equation*}

For a $k$-rectifiable set $M\subset \R^n$, for  $\mathcal{H}^k$-a.e. $x\in M$ there is measure theoretic tangent space $T_xM$. A map $\tau:M\to \Lambda_k(\R^n)$ is an orientation on $M$ if $\tau$ is $\mathcal{H}^k$-measurable and for $\mathcal{H}^k$a.e. $x\in M$, $\tau(x)$ is a unit simple $k$-vector than spans $T_xM$. 
For a $k$-rectifiable set $M\subset\Omega$, $\tau$ an orientation on $M$ and $\rho:M\to\R^+$ a locally $\mathcal{H}^k$-summable function, we define the rectifiable $k$-current $T:=\llbracket M,\tau,\rho\rrbracket\in \mathcal{D}_k(\Omega)$ as:

\begin{equation*}
    \langle T,\omega\rangle:=\int_M\langle \omega, \tau\rangle\rho\,{\rm d}\mathcal{H}^k\,,\quad \forall \omega\in \mathcal{D}^k(\Omega)\,.
\end{equation*}
The function $\rho$ is called multiplicity of $T$.
We denote by $\mathcal R_k(\Omega)$ the set of rectifiable k-currents, and by $\mathcal I_k(\Omega)$ the set of integer rectifiable k-currents, i.e., the set of rectifiable k-currents with integer-valued multiplicity $\rho$. A current $T\in\mathcal{I}_k(\Omega)$ with $\partial T\in\mathcal{I}_{k-1}(\Omega)$ is called integral.

In the context of graphs over sets in $\R^n$ it is convenient to identify $\R^{n+1}=\R^n_x\times\R_y$ for which the standard basis is $(e_1,\dots,e_n,e_y)$ and the corresponding coordinates are $(x,y)=(x_1,\dots,x_n,y)$.



\subsubsection{$BV$ functions over currents}

For a rectifiable  $k$-current  $T=\llbracket M,\tau,\rho\rrbracket$ and a function $u:M\to\R$ we introduce the set between the graph of $u$ and $0$:
\begin{equation*}
E_{u,T} :=\{   (x,y)\in M\times \R: 0<y<u(x) \textrm{ if } 0<u(x)\,, u(x)<y<0 \textrm{ if } u(x)<0 \}.
\end{equation*}
and for every $(x,y)\in E_{u,T}$ an induced orientation:
\begin{equation*}
 \alpha(x,y):=
   \begin{cases}
   -e_y\wedge \tau(x) &\textrm{ if } y>0, \\
   e_y\wedge \tau(x) &\textrm{ if } y<0.
\end{cases} 
\end{equation*}
We define the $k+1-$current $\Sigma_{u,T}:=\llbracket E_{u,T},\alpha,\rho\circ p\rrbracket$ with $p(x,y)=x$ for every $(x,y)\in\R^n\times\R$ and we obtain the generalized graph of $u$ over $T$:
\begin{equation*}
    G_{u,T}:=-\partial\Sigma_{u,T}+T\otimes\delta_0.   
\end{equation*}
where $T\otimes\delta_0$ is defined as $\langle T\otimes\delta_0 ,\omega\rangle=\langle T,\omega(\cdot,0)\rangle.$
 We now recall the definition of $BV$ functions over integer rectifiable currents \cite[Definition 2.5]{ADS}.
\begin{definition}
    We consider  a rectifiable  $k$-current $T=\llbracket M,\tau,\rho\rrbracket$ and $u:M\to\R$. We say that $u$ is a function of bounded variation over $T$ if the mass $M(G_{u,T})$ of the generalized graph is finite.
    The set of the functions of finite bounded variation over $T$ is denoted by $BV(T)$. Moreover, we denote by $BV(T;A)$ the set of functions in $BV(T)$ taking values in $A\subset\R$ $\mathcal{H}^k$ a.e..
\end{definition}
 

\subsubsection{Sobolev spaces with respect to currents}
We define Sobolev spaces with respect to currents based on the above definition of $BV$ spaces. Let $S=\llbracket M,\tau,\rho\rrbracket $ be as above, and $\mu=\|S\|$.  For $\H^k\res M$ almost every $x\in M$, we may define  $P(x)$ as the  the projection onto the tangent space of $M$ at $x$. For such $x$ we set 
\[
\nabla_{\mu} u(x):=P(x)\nabla u(x)\,.
\]
For $p\in [1,\infty)$ and $u\in C^\infty_c(\R^n)$, we define
\[
\|u\|_{H^{1,p}(S)}=\|u\|_{L^p_\mu(\R^n)}+\|\nabla_\mu u\|_{L^p_\mu(\R^n)}\,.
\]
The Sobolev space $H^{1,p}(S)$ is defined as the closure of $C^\infty_c(\R^n)$ with respect to the norm $\|u\|_{H^{1,p}(S)}$. 

\begin{remark}
  The definition of Sobolev spaces with respect to measures is a delicate issue that we will not treat in any depth  here; we refer the interested reader to \cite{BBS,BBF,FM}. Our definitions (the one of $H^{1,p}(S)$ above, and the definition of $\Hu^{1,p}_{\mu_\e}(\R^n)$ below) are slightly different to the one used in these references and we cannot apply the theory developed there. The only property of these spaces  that we will prove  (paralleling Proposition 2.1 in \cite{BBS}) is the uniqueness of the tangential gradient $\nabla_{\|S\|}v$ for a given $v\in H^{1,p}(S)$ or of the gradient $\nabla v$ for a given $v\in \Hu^{1,p}_{\mu_\e}(\R^n)$, which does not immediately follow from the definitions. As a straightforward consequence, one obtains that the spaces $H^{1,p}(S),\Hu^{1,p}_{\mu_\e}(\R^n)$ are reflexive. However, since the only fact that will be of importance for us is that every element in these spaces can be approximated by smooth functions (a fact that is implicitly exploited by appealing to \cite[Theorem 1]{OR} in the proof of Proposition \ref{prop:compactness} below), we have relegated the proof of the uniqueness of the gradient to the appendix, see Lemmata \ref{lem:gradient_uniqueness_current} and \ref{lem:gradient_uniqueness_bulk}.
\end{remark}


\subsection{Oriented Varifolds} We use the notations and definitions of \cite[Section 2.2]{OR} and \cite{Hut}. We denote the set of $k-$dimensional oriented subspaces of $\R^n$ by $G^o(n,k)$. We can identify this set with the simple unit $k$-vectors in $\Lambda_k(\R^n)$. 
\begin{definition}
    An oriented $k$-varifold $V$ is an element of $\M(\R^n\times G^o(n,k)) $:
    \begin{equation*}
        V(\varphi)=\int_{\R^n\times G^o(n,k)}\varphi(x,\xi)\d V(x,\xi)
    \end{equation*}
    for every $\varphi\in C^0_c(\R^n \times G^o(n,k))$.
\end{definition}

For a $k-$rectifiable set $M$ with orientation $\tau$ and $\theta_\pm:M\to \R^+$ locally $\H^k$-summable such that $\theta_++\theta_->0$ we define $\underline{v}(M,\tau, \theta_\pm)$ as the following $k$-dimensional rectifiable oriented varifold:
\begin{equation*}
    \underline{v}(M,\tau, \theta_\pm)(\varphi)=\int_M \big( \theta_+(x)\varphi(x,\tau(x))+\theta_-(x)\varphi(x,-\tau(x))\big)\d\H^k\,.
\end{equation*}
If the multiplicity functions $\theta_\pm$ are $\N$-valued, then $\underline{v}(M,\tau, \theta_\pm)$ is an integral oriented $k-$varifold. We denote the set of $k$-dimensional  oriented varifolds by $\sfV^o_k(\R^n)$, the set of $k$-dimensional oriented rectifiable varifolds by $\sfRV_k^o(\R^n)$, and the set of $k$-dimensional oriented integral varifolds by $\sfIV^o_k(\R^n)$.

To an  oriented $k$-varifold $V$ we can associate the $k-$current:
\begin{equation*}
    \underline{c}(V)(\varphi)=\int_{\R^n\times G^o(n,m)}\langle\varphi(x),\xi\rangle\d V(x,\xi)\,.
\end{equation*}
Hence, the  convergence as oriented varifolds implies the convergence of the associated currents.

The first variation of a varifold $V\in \sfV^0_k(\R^n)$ is the $\R^n$ valued distribution $\delta V$ defined by
\[
\delta V(X)=-\int \nabla X:P_T \d V(x,T)\quad \text{ for } X\in C^1_c(\R^n;\R^n)\,,
\]
where $P_T$ denotes  projection to the tangent plane orthogonal to $T\in G^o(n,k)$ in matrix form. If there exists $H_V\in L^1_{\|V\|}(\R^n;\R^n)$ such that
\[
\delta V(X)=\int H_V\cdot X \d\|V\|\,,
\]
then we say that $V$ possesses generalized mean curvature $H_V$. If $A\subset \R^n$ is $k$-rectifiable and $\|V\|=\H^{k}\res A$ for some $V\in \sfIV_k^o(\R^n)$ possessing generalized mean curvature $H_V$, then we also write $H_A\equiv H_V$.

\subsection{Measure-function pairs}
  We recall the definition of measure-function pairs from \cite{Mos01} (see also \cite{Hut}). Let $\Omega\subset\R^n$. If $\mu\in\mathcal M(\R^n)$ and $f\in L^1_{{\rm loc},\mu}(\Omega;\R^m)$, then we say that $(\mu,f)$ is a measure-function pair over $\Omega$ with values in $\R^m$.
  \begin{definition}[Convergence of measure-function pairs]\label{def:measure-function-pairs} Let $\{(\mu_k,f_k)\colon k\in\mathbb{N}\}$ and $(\mu,f)$ be measure-function pairs over $\Omega$ with values in $\R^m$, and $1\le q<\infty$. 
  \begin{enumerate}[label=$(\roman*)$]
      \item We say that $(\mu_k,f_k)$ converges weakly in $L^q$ to $(\mu,f)$ and write 
      \begin{equation*}
          (\mu_k,f_k)\rightharpoonup(\mu,f)\quad\text{ in } L^q
      \end{equation*}
      if $\mu_k\stackrel{*}{\rightharpoonup}\mu$ in $\mathcal{M}(\Omega)$, $\mu_k\res f_k\stackrel{*}{\rightharpoonup}\mu\res f$ in $\mathcal{M}(\Omega;\R^m)$, and $\|f_k\|_{L^p_{\mu_k}(\Omega;\R^m)}$ is uniformly bounded.
      \item   We say that $(\mu_k,f_k)$ converges strongly in $L^q$ to $(\mu,f)$ and write 
  	\begin{equation*}
  		(\mu_k,f_k)\to(\mu,f)\quad\text{ in } L^q
  	\end{equation*}
  if for all $\varphi\in C_c^0(\R^n\times\R^m)$,
  \begin{equation*}
\lim_{k\to\infty}\int_{\Omega}\varphi(x,f_k(x))\d\mu_k(x)=\int_\Omega\varphi(x,f(x))\d\mu(x)\,,
  \end{equation*}
and
\begin{equation*}
\lim_{j\to\infty}\int_{\S_{k,j}}|f_k|^q\d\mu_k=0\quad\text{ uniformly in }k\,,
\end{equation*}
where $\S_{k,j}=\{x\in\Omega\colon |x|\ge j\text{ or }|f_k(x)|\ge j\}$.
  \end{enumerate}

  \end{definition}
  %


\section{Setting of the problem and first result}
\label{sec:main_results}
In this section we describe the setting of the problem we are considering. Afterwards we state and prove our first main result: Theorem \ref{thm1}.

We let $W\colon\R\to\R$ be a double well potential, i.e., a continuous function, such that for some $T,c>0$, $p>1$ it holds
\begin{equation}\label{hp:W}
W^{-1}(0)=\{-1,1\}\,,\quad c|t|^p\le W(t)\le \frac1c|t|^p\quad \forall |t|\ge T\,.
\end{equation}
Let $\phi:\R\to\R$ be defined by
\begin{equation}\label{phi}
\phi(t)\defas\int_{-1}^t\sqrt{2W(s)}\d s\,, 
\end{equation}

and set 
\begin{equation}\label{def:c_W}
   \sigma=\phi(1)\,.  
\end{equation}

For $\e>0$, let $u_\e\in W^{1,2}_{\loc}(\R^n)$ be such that
\begin{equation}\label{mue}
   \mu_\e:=\left(\frac{\e}{2}|\nabla u_\e|^2+\e^{-1}W(u_\e)\right)\L^n 
\end{equation}

is a finite measure.

For $p\in[1,\infty)$ and $v\in C^\infty_c(\R^n)$, set
\[
\|v\|_{\Hu^{1,p}_{\mu\e}(\R^n)}:=\|v\|_{L^p_{\mu_\e}(\R^n)}+\|\nabla v\|_{L^p_{\mu_\e}(\R^n;\R^n)}\,.
\]
We define $\Hu^{1,p}_{\mu_\e}$ as the completion of $C^\infty_c(\R^n)$ with respect to $\|v\|_{\Hu^{1,p}_{\mu\e}(\R^n)}$. In Lemma \ref{lem:gradient_uniqueness_bulk} in the appendix we show that this definition yields a unique gradient $\nabla v$ for every $v\in H^{1,p}_{\mu_\e}(\R^n)$.

The above definition of the Sobolev spaces with respect to the measure $\mu_\e$ has been chosen such that we may perform a slicing procedure with respect to the level sets of $u_\e$, and obtain on almost every slice  a Sobolev function on the associated current, see Lemma \ref{lem:Hdefslicing} below.

For $n\ge2$, $\eps>0$ we consider the family of functionals $I_\eps\colon(L^1_\loc(\R^n))^2\to[0,+\infty]$ defined by
\begin{equation}\label{def:I_e}
	I_\eps(u,v)\defas \int_{\R^n} \left(
	\frac{W(v)}\eps+ \frac\eps{ 2}|\nabla v|^2
	\right)
	\left(
	\frac{W(u)}\eps+ \frac\eps{ 2}|\nabla u|^2
	\right)\dx \,,
\end{equation}
if $u\in W^{1,2}_\loc(\R^n)$,  and $v\in \Hu^{1,2}_{\mu_\eps}(\R^n)$,  and $+\infty$ otherwise.
 We denote also by $M_\eps\colon L^1_\loc(\R^n)\to[0,+\infty]$ the classical Modica-Mortola functional, i.e.,
\begin{equation}\label{def:M_e}
M_\eps(u)\defas\int_{\R^n}\left(\frac{W(u)}{\eps}+\frac\eps{ 2}|\nabla u|^2\right)\dx=\mu_\e(\R^n)\,,
\end{equation}
if $u\in W^{1,2}_\loc(\R^n)$, and $+\infty$ otherwise.\\
In this model  the variable $u$ has to be understood as a regularization of a  piecewise constant function  $2\chi_E-1\in BV_{\rm loc}(\R^n;\{\pm1\})$  where the surface $\partial E$ represents a bio-membrane, whereas the variable $v$ is a phase-field variable modeling the phase separation on the membrane itself.\\

  \noindent

  We now state the first main result of this paper.

  \begin{theorem}[Lower bound and compactness of $I_\eps$]\label{thm1}
  	Let $I_\eps$ be defined as in \eqref{def:I_e}. Then the following hold:
  	    	\begin{enumerate}[label=$(\arabic*)$]
  		\item Compactness.	Let $((u_\eps,v_\eps))_{\eps>0}\subset (L^1_\loc(\R^n))^2$ be a sequence such that  $\sup I_\eps(u_\eps,v_\eps)<+\infty$. Assume also that there exists $u\in BV_\loc(\R^n;\{{ -1},1\})$ such that $u_\eps\to u$ strictly in $BV_\loc(\R^n)$ in the sense of Definition \ref{StrictBVConc}.
  		Then there exists  $v\in  BV(S;\{{ -1},1\})$, with $S=\llbracket J_u,\star\nu_u,1\rrbracket$, such that, up to subsequence, 
  \begin{equation}\label{eq:81}
    \left(|\nabla (\phi\circ u_\e)|\L^n,v_\e\right)\to (\sigma \H^{n-1}\res J_u, v)\quad\text{ in }L^q
  \end{equation}
as measure-function pairs for every $q\in [1,p)$,  where $\phi$ is given in \eqref{phi}. 

  		\item Lower bound. Let $((u_\eps,v_\eps))_{\eps>0}\subset (L^1_\loc(\R^n))^2$ be a sequence and  $(u,v)\in BV_\loc(\R^n;\{{ -1},1\})\times BV(S;\{{ -1},1\})$ with $S=\llbracket J_u,\star \nu_u,1\rrbracket$, such that $u_\e\to u$ strictly in $BV$ and  the convergence \eqref{eq:81} holds. 
Then 
  		\begin{equation}\label{liminf-inequality-1}
  			\liminf_{\eps\searrow0}I_\eps(u_\eps,v_\eps)\ge I(u,v)\,,
  		\end{equation}
  		where $I\colon (L^1_{\rm loc}(\R^n))^2\to[0,+\infty]$ is defined as
  		\begin{equation}\label{def:I}
      			I(u,v)\defas \begin{cases}\sigma^2 \mathcal{H}^{n-2}(J_v)& \text{ if }(u,v)\in BV_{\rm loc}(\R^n;\{{ -1},1\})\times BV(S;\{{ -1},1\}\\ +\infty &\text{ else,} \end{cases} 		\end{equation}
    where $J_v=\supp \partial \llbracket \{v=1\},\star \nu',1\rrbracket$, and   		$\sigma$ is as in \eqref{def:c_W}. 
  	\end{enumerate}
  \end{theorem}

\begin{remark}
    The reader may wonder why in the compactness part, we do not show compactness for the measure-function pairs 
    \begin{equation}
      \label{eq:diffusemeas}
    \left(\left(\frac{\e}{2}|\nabla u_\e|^2+\e^{-1}W(u_\e)\right)\L^n,v_\e\right)\,.   
    \end{equation}
   
    In a way, these might be considered as the more natural objects, since one can think of them as the diffuse approximations of the sharp interface measure. The reason why  we show the compactness for the measure-function pairs 
     $\left(|\nabla (\phi\circ u_\e)|\L^n,v_\e\right)$ as in \ref{eq:81} is that
    the latter allows for a straightforward slicing operation via the coarea formula. On these slices we apply the results from \cite{OR}, which are based on the compactness result \cite[Theorem 4.1]{ADS}, which in turn is based on the Federer-Fleming compactness theorem \cite[Theorem 27.3]{Simo83}. It might be possible to derive an analogous result for the measures in \eqref{eq:diffusemeas}, but we expect that this would require an argument that mirrors the Federer-Fleming theorem, showing the convergence of the diffuse currents to an integral one. We believe that our way of relying on the existing literature is more efficient. 
\end{remark}

\begin{proof}[Proof of Theorem \ref{thm1}] The proof is the consequence of two  upcoming propositions.
 Proposition \ref{prop:compactness} gives the compactness part of the theorem. Moreover, we obtain that for a.e. $t\in [0,\sigma]$:
    \begin{equation*}
        \int_{\partial_* E_\eps^t} \varphi(x,v_\eps)\dHn\to \int_{J_u} \varphi(x,v)\dHn \text{ for every }\varphi\in C^0_c(\R^n\times\R)\,,
    \end{equation*}
    where $E_\eps^t=\{\phi\circ u_\eps > t\}$. This additional property is the assumption of Proposition \ref{prop:lower_bound}, which provides the lower bound.
\end{proof}

\subsection{Proof of compactness and lower bound}
\label{sec:proofthm1}
We first recall \cite[Theorem 1]{OR} which will be used to obtain compactness in the proof of Proposition \ref{prop:compactness} once we have performed a slicing procedure.
\begin{theorem}[\cite{OR} Theorem 1]
\label{thm1OR}
    Let a family $(E_\eps)_{\eps>0}$ of finite perimeter sets in $\R^n$, whose boundaries carry the currents $S_\eps=\llbracket \partial_*E,\star\nu_{E_\eps},1\rrbracket$  and a sequence $(v'_\eps)_{\eps}\in H^{1,2}(S_\eps)$ be given.

    Assume that for some set $E$ of finite perimeter  $\chi_{E_\eps}\to\chi_E$ strictly in $BV(\R^n)$. Let $\nu'=\nu_E:\partial_*E\to \Sph^{n-1}$ denote the inner normal of $E$ and set $S'=\llbracket \partial_*E,\star\nu',1\rrbracket$, $\mu=\H^{n-1}\res \partial_*E$. Let us further assume that for some $\Lambda'>0$
    \begin{equation*}
        \int_{\R^n}\bigg(\frac{W(v_\eps')}{\eps}+\frac{\eps}{2}|\nabla v_\eps'|^2\bigg){\rm d}\mu'_\eps\le\Lambda \,.
    \end{equation*}
    where $\mu'_\eps=\H^{n-1}\res \partial_*{E_\eps}$.
    Then there exists $v\in BV(S;\{-1,1\})$ and a subsequence $\eps\to 0$ such that the following holds:
    \begin{equation*}
        (\mu_\eps',v_\eps')\to (\mu,v) \text{ as  measure-function pairs in } L^q \text{ for any } 1\leq q< p\,.
    \end{equation*}
    Moreover, $\llbracket \{v=1\},\star\nu',1\rrbracket$ is an integral $(n-1)-$current and we have the lower bound estimate 
    \begin{equation*}
        \liminf_{\eps\searrow0} \int_{\R^n}\bigg(\frac{W(v_\eps')}{\eps}+\frac{\eps}{2}|\nabla v_\eps'|^2\bigg){\rm d}\mu'_\eps\ge \sigma \H^{n-2}(J_v)\,.
    \end{equation*}

\end{theorem}

As a preparatory step, we state and prove a lemma that clarifies the relation between the different notions of Sobolev spaces over currents and measures that we are using:

\begin{lemma}
\label{lem:Hdefslicing}
Let $\e>0$, $u_\e\in W^{1,2}_\loc(\R^n)$ such that the corresponding $\mu_\e$ defined in \eqref{mue} is a finite measure, $E_\e^t=\{x\in\R^n:u_\e(x)>t\}$, and  
\[
S_{\e,t}:=\Big\llbracket \partial_*\{u_\e>t\},\star\frac{\nabla u_\e}{|\nabla u_\e|},1\Big\rrbracket\,.
\]
If  $v\in \Hu^{1,2}_{\mu_\e}(\R^n)$, then $v\in H^{1,2}(S_{\e,t})$ for almost every $t$.    
\end{lemma}

\begin{proof}
Assume that $v_j\in C_c^\infty(\R^n)$ with $v_j\to v$, $\nabla v_j\to \nabla v$ in $L^2_{\mu_\e}(\R^n)$. Then writing $U_\eps=\phi\circ u_\eps$, we have
\[
\begin{split}
    |\nabla U_\eps|&=|\phi'(u_\e)||\nabla u_\e|\\
    &=\frac{\sqrt{2W(u_\e)}}{\sqrt{\e}}\sqrt{\e}|\nabla u_\e|\\
    &\le \frac{\e}{2}|\nabla u_\e|^2+\e^{-1}W(u_\e)\,,
\end{split}
\]
where we have used the Cauchy-Schwarz inequality in the last line. (We will refer to this estimate as the ``Modica-Mortola trick'' in the sequel.)
Hence
\[
\begin{split}
0&=\lim_{j\to\infty}\int |v_j-v|^2+|\nabla v_j-\nabla v|^2\d\mu_\e\\
&\geq \lim_{j\to\infty}\int_{\R^n} \left(|v_j-v|^2+|\nabla v_j-\nabla v|^2\right) |\nabla U_\e(x)|\d x\\
&\geq \lim_{j\to\infty}\int_{-\infty}^\infty\int_{\partial_*\{U_\eps>t\}}( |v_j-v|^2+|\nabla_{\mu_{\e,t}} v_j-\nabla_{\mu_{\e,t}} v|^2 )\d\H^{n-1}\d t\,\,
\end{split}
\]
where $\mu_{\e,t}=\|S_{\e,t}\|$.  Our claim follows from the fact that $\phi$ is a homeomorphism.
\end{proof}

As a further preparation of the proof of the compactness statement in Theorem \ref{thm1}, we introduce two preliminary lemmata. The first one will serve to rearrange the slices $S_{\e,t}$ such that we obtain uniform boundedness in $\e$ of the Modica-Mortola energy for a fixed slice $t$. The second one assures that convergence for almost every slice suffices to obtain the same limit on every slice.

\begin{lemma}\label{Rear}
   Let $\Lambda>0$ and let  $G_k:[a,b]\to[0,\infty)$ be a sequence of measurable functions with $k\in\N$ such that
    \begin{equation*}
        \sup_{k\in\N} \int_a^b G_k(t)\dt\le \Lambda.
    \end{equation*}
    Then, for every $k\in\N$, there exists $h_k:[a,b]\to [a,b]$ with the following properties:
\begin{itemize}
    \item[(i)] For every $t_0\in(a,b)$, there exists $C(\Lambda,t_0)>0$ such that  
\begin{equation}\label{prop1}
     \sup_{k\in\N} G_k(h_k(t))\leq C(\Lambda,t_0)\quad\text{ for a.e. }t\in [a,t_0]\,.
\end{equation}
\item[(ii)] For every $t_0\in(a,b)$ there exist $a<s_1<s_2<b$ such that
\begin{equation}\label{prop2}
   s_1<h_k(t_0) <s_2\quad \text{ for every }k\in\N\,,\,. 
\end{equation}
\item[(iii)] For every non-negative measurable function $f:[a,b]\to \R$,
    \begin{equation}\label{prop3}
    \int_a^b f(h_k(t))\dt=\int_a^b f(t)\dt \,.
    \end{equation}
\end{itemize}    
\end{lemma}

\begin{proof} 
For simplicity of notation, we will prove the statement for  $[a,b]=[0,1]$.

We fix $0<\delta<\frac{1}{2}$ and  $k\in \N$ and  use a recursive argument. Since 
\begin{equation*}
    \int_0^1 G_k(t)\dt\le \Lambda
\end{equation*}
there exists a set $A_k^0\subset [\frac{\delta}{4},1-\frac{\delta}{4}]$ such that
\begin{equation*}
    |A_k^0|>1-\delta\quad \text{ and }\quad \forall\, t\in A_k^0\,, \quad G_k(t)\leq \frac{3\Lambda}{\delta}\,.
\end{equation*}
Hence, we can find a compact set $B_k^0\subset [\frac{\delta}{4},1-\frac{\delta}{4}]$ such that
\begin{equation*}
     |B_k^0|=1-\delta\quad \text{ and } \quad\forall\, t\in B_k^0\,,\quad G_k(t)\leq \frac{3\Lambda}{\delta}\,.
\end{equation*}
We then define the function:
\begin{align*}
h_k^0\colon &[0,1-\delta) \longrightarrow B_k^0\\
&t\longmapsto \inf\{s\in B_k^0\textrm{, } |\{x\in B_k^0 \textrm{, } x\leq s\}|\geq t\}.
\end{align*}
Analogously, there exists a compact set $B_k^1 \subset [-\frac{\delta}{8},1-\frac{\delta}{8}]$ such that:
\begin{equation*}
B_k^0\cap B_k^1=\emptyset\,\textrm{,}\quad|B_k^1|=\frac{\delta}{2} \quad\textrm{ and }\quad  \forall\, t\in B_k^1\,, \quad G_k(t)\leq \frac{5\Lambda}{\delta} \,.
\end{equation*}
We define the function:
\begin{align*}
h_k^1\colon &[1-\delta, 1-\frac{\delta}{2})  \longrightarrow B_k^1\\
&t\longmapsto \inf\{s\in B_k^1\textrm{, } |\{x\in B_k^1 \textrm{, } x\leq s\}|\geq t-(1-\delta)\}.
\end{align*}
Carrying on in the same way, we can find a family of sets $\{B_k^i\}_{i\in\N}$ and a family of functions $\{h_k^i\}_{i\in\N}$ such that:
\begin{itemize}
\item $\{B_k^i\}_{i\in\N}$ is a family of pairwise disjoint compact subsets of $[0,1]$ with $|\cup_iB_k^i|=1$,
satisfying
\begin{equation*}
	|B_k^0|=1-\delta\,,\quad |B_k^i|=\frac\delta{2^{i}},\quad B_k^i\subset [\frac{\delta}{2^{i+2}},1-\frac{\delta}{2^{i+2}}]\quad \forall i\ge 1\,,
\end{equation*} 
and 
\begin{equation*}
	G_k(t)\leq \frac\Lambda\delta K_i
	\quad\forall t\in B_k^i\,,
\end{equation*}
with  $K_i:= (2^{i+1}+1)$;
\item $\{h_k^i\}_{i\in\N}$ is a family of functions  
\begin{equation*}
	h_k^0\colon[0,1-2\delta)\to B_k^0\,,\quad 	h_k^i\colon[1-\frac\delta{2^{i-1}}, 1-\frac\delta{2^{i}})\to B_k^i\quad \forall i\ge1
\end{equation*}
defined as
\begin{equation*}
\begin{split}
h^0_k(t)&:= \inf\{s\in B_k^0\textrm{, } |\{x\in B_k^0 \textrm{, } x\leq s\}|\geq t\}\,,\\
h^i_k(t)&:= \inf\{s\in B_k^i\textrm{, } |\{x\in B_k^i \textrm{, } x\leq s\}|\geq t-(1-\frac\delta{2^{i-1}})\}\,.
\end{split}
\end{equation*}
\end{itemize}
Hence, we  define  the function $h_k\colon[0,1)\to[0,1]$ as 

\begin{equation*}
h_k(t):=\begin{cases}
	h_k^0(t)& \text{if }t\in [0,1-\delta)\\[4pt]
	h_k^i(t)&\text{if }t\in [1-\frac\delta{2^{i-1}}, 1-\frac\delta{2^{i}})\,, i\ge1
\end{cases}\,.
\end{equation*}

Thanks to the definition of $h_k$, for every $0<t_0<1$, there exists $i\in\N$ such that $t_0< 1-\frac{\delta}{2^i}$. Hence, for every $k\in\N$ and every $t\in [0,t_0]$ we have that $$G_k(h_k(t))\leq K_i\frac{\Lambda}{\delta}:=C(\Lambda,t_0)\,.$$
By construction, \eqref{prop2} is satisfied.
Moreover, for every $s\in[0,1]$, \begin{equation*}
    \begin{split}
        |\{t\in [0,1]\textrm{, } h_k(t)<s\}|&=\sum_{i\in N} |\{t\in [0,1]\textrm{, } h_k(t)<s \textrm{ and  } h_k(t)\in B_k^i\}|\\
        &=\sum_{i\in N}|B_k^i\cap [0,s]|=s\,.
    \end{split}
\end{equation*}
Thus, by density, for every measurable function $f:(0,1)\to\R_+$, we also deduce the third property.
\end{proof}

In the following lemma, we will consider a set $X$ with a notion of convergence denoted by ``$\to$''. All that we require of the ``notion of convergence'' is that it maps the set of sequences in $X$ to the space $X$ with an  additional symbol, reserved for non-convergent sequences, such that the following two properties are fulfilled:
\begin{itemize}
    \item[(i)] If $x_k\to x$, then $x_{k_l}\to x$ for every strictly increasing sequence $(k_l)_{l\in\N}\subset\N$
    \item[(ii)] If $x_k\not \to x$, then there exists a strictly increasing sequence $(k_l)_{l\in\N}\subset\N$ such that no subsequence $(k_{l_m})_{m\in\N}$ satisfies $x_{k_{l_m}}\to x$.
\end{itemize}
 
\begin{lemma}\label{ConvRear}
Let $X$ be a   space with a notion of converging sequence  and  $F_k:[a,b]\to X$ with $k\in\N$ such that:
\begin{enumerate}
    \item For a.e. $t\in[a,b]$ and for every strictly monotone sequence $(k_l)_{l\in\N}\subset \N$,
there exist $F_t\in X$ and a subsequence $(k'_l)_{l\in\N}\subset \N$ of $(k_l)_{l\in\N}\subset \N$ such that $ F_{k'_l}(t)\to F_t\in X$.
\item If for a strictly monotone sequence $(k_l)_{l\in\N}\subset \N$ we have that
\begin{equation}\label{UniquenessLimitSlice}
     F_{k_l}(t_1)\to F_{t_1}\quad \text{ and }\quad
         F_{k_l}(t_2)\to F_{t_2}\,,
    \end{equation}
    then $F_{t_1}=F_{t_2}$.
\end{enumerate}

 Then there exist $F\in X$ and a strictly increasing sequence $(k_l)_l$ such that for a.e. $t\in[a,b]$:
    \begin{equation*}
        F_{k_l}(t)\to F\,.
    \end{equation*}
\end{lemma}

\begin{proof}
   We choose a strictly monotone sequence $(k_l)_{l\in\N}\subset \N$ such that $ F_{k_l}(t)\to  F_t$ for a specific $t\in[a,b]$. We assume by contradiction that there exists a set of positive measure $\Sigma\subset [a,b]$ such that for every  $s\in\Sigma$, $  F_{k_l}(s)$ is not converging to $ F_t$. Thus, for a.e. $s\in\Sigma$ there exists a subsequence $(k^{s}_l)_{l\in\N}\subset \N$ of $(k_l)_{l\in\N}\subset \N$, such that for no further subsequences of $( F_{k^{s}_l}(s))_{l\in\N}$ is converging to $ F_t$. But, by assumption, for a.e. $s\in\Sigma$, we can extract from $(F_{k^{s}_l}(s))_{l\in\N}$ a subsequence converging to $ F_s\in X$.
    By \eqref{UniquenessLimitSlice}, this is a contradiction.

\end{proof}

In the upcoming proposition, we will label sequences differently from the rest of the paper. Whereas elsewhere, we write $(y_\e)_{\e}$ for sequences and tacitly assume that the real parameter $\e>0$ represents a sequence $\e_k\downarrow 0$, we make this sequence explicit in the proposition. We do so in order  to avoid confusion about our choices of subsequences etc.

\begin{proposition}[Compactness]\label{prop:compactness}
Let $\e_k\downarrow 0$, and  $(u_{\eps_k},v_{\eps_k})_k\subset (L^1(\R^n))^2$ be a sequence such that  
\[
\sup_{k\in\N} I_{\eps_k}(u_{\eps_k},v_{\eps_k})<+\infty\,.
\]
Assume also that there exists $u\in BV(\R^n;\{-1,1\})$ such that $u_{\eps_k}\to u$ strictly in $BV_{\loc}(\R^n)$.
Then there exists  $v\in  BV(S;\{-1,1\})$, with $S=\llbracket J_u,*\nu_u,1\rrbracket$, such that, up to subsequence,  
  \begin{equation}
  \label{eq:comp_str_mfp}
    \left(|\nabla (\phi\circ u_{\eps_k})|\L^n,v_{\eps_k}\right)\to (\sigma \H^{n-1}\res J_u, v)\quad\text{ in }L^q
  \end{equation}
as measure-function pairs for every $q\in [1,p)$, with $\phi$ as in \eqref{phi}. Moreover, we also have for a.e. $t\in[0,\sigma]$ that 
\begin{equation}
\label{eq:comp_str_mfp2}
        \int_{\partial_* E_\eps^t} \varphi(x,v_\eps)\dHn\to \int_{J_u} \varphi(x,v)\dHn \text{ for every }\varphi\in C^0_c(\R^n\times\R)\,,
    \end{equation}
    where $E_\eps^t=\{\phi\circ u_\eps > t\}$.

\end{proposition}
\begin{proof}
We restrict ourselves to a subsequence (that we do not relabel) such that $I_{\eps_k}(u_{\eps_k},v_{\eps_k})$
converges to the $\liminf$ in \eqref{liminf-inequality-1}.

In order to alleviate the notation, we will write $u_k$ and $v_k$ instead of $u_{\eps_k}$ and $v_{\eps_k}$,
\[
U_k\defas\phi\circ u_{k}\,.
\]
For the superlevel sets of $U_k$, we introduce the notation
\begin{equation}\label{level-sets}
    E_k^s\defas\{x\in\R^n\colon U_k>s\}\,.
\end{equation}

\smallskip

\noindent
{\textbf{Step 1: Slicing over $u_k$.}} By the chain rule and the Cauchy-Schwarz inequality with $\e_k$, 
\[
|\nabla U_k(x)|\leq \sqrt{2W(u_k(x))}|\nabla u_k(x)|\leq \frac{\e_k}{2}|\nabla u_k|^2+\frac1{\e_k}W(u_k(x))\,.
\]
Hence, using the coarea formula, 
\begin{equation}\label{eq:5}
	\begin{split}
	  \Lambda\ge	I_{\eps_k}(u_k,v_k)&\ge 
		\int_{\R^n }
		\left(
		\frac{W(v_k)}{\eps_k}+ \frac{\eps_k}{ 2}|\nabla v_k|^2
		\right) |\nabla U_k|\dx\\
&=\int_{-\infty}^{+\infty}\int_{\partial_* E_{k}^s} \left(
      \frac{W(v_k)}{\eps_k}+ \frac{\eps_k}{ 2}|\nabla v_k|^2
    \right)\dHn\ds\,.
			\end{split}
\end{equation}

\smallskip
\noindent

{\textbf{Step 2: Equality of the limit on slices.}}
We begin with the following observations:
\begin{enumerate}[label=\textbf{O.\arabic*}]

    \item  By the coarea formula \ref{CoAreaBV}, for a.e. $t\in [0,\sigma]$, the set $E^t_{k}$ is a set of finite perimeter for every $k\in\N$\,.
    \item \label{StrictConvSlice}  For a.e. $t\in[0,\sigma]$,  $\chi_{E^t_k}\to {\frac{u+1}{2}}$ in $L^1_{\rm loc.}(\R^n)$. By lower semicontinuity of the perimeter and the fact that 
\begin{equation*}
      \int_0^\sigma {\rm Per}(E^t_k)\dt \to  \sigma\H^{n-1}(J_u)\,,
\end{equation*}
we have that $E^t_k$ converges strictly in $BV_{\rm loc}(\R^n)$ to $\chi_{\frac{u+1}{2}}$ for a.e. $t\in[0,\sigma]$.
    \item \label{LevelSet} By \cite[Theorem 1.1]{MSZ}, for a.e. $t\in [0,\sigma]$,
    $\H^{n-1} (U_k^{-1}(t) \Delta \partial_*E^t_k  )=0$ for every $k\in\N$ for the precise representative of the Sobolev map $U_k$\,.
    \end{enumerate}

We define $\S\subset[0,\sigma]$ as the set of $t\in [0,\sigma]$, such that one of the previous properties does not hold. Hence,  we have that $|\S|=0$. 

We claim the following: If, for some $t_1,t_2\in [0,\sigma]\backslash \S$, $t_1<t_2$,  there exists a  subsequence $(k_l)_{l\in\N}$ and $v^{t_1},v^{t_2}\in  BV(S; \{-1,1\})$ such that
\begin{equation}\label{strong-t1-t2}
(\mathcal{H}^{n-1}\res\partial_*E^{t_i}_{k_l},v_{k_l})\to (\mathcal{H}^{n-1}\res J_u,v^{t_i})\text{ in } L^q\,, \ \forall \,1\le q< p\,, \quad \text{ for }l\to\infty
\end{equation} as measure-function pairs
for $i=1,2$, then  $v^{t_1}=v^{t_2}$ for $\mathcal{H}^{n-1}$ a.e. $x\in J_u$.
In order to alleviate the notation, we will assume that the subsequence is the sequence itself, such that the assumption becomes 
\begin{equation*}
(\mathcal{H}^{n-1}\res\partial_*E^{t_i}_{k},v_{k})\to (\mathcal{H}^{n-1}\res J_u,v^{t_i})\quad \text{  in } L^q\,, \ \forall \,1\le q< p\,, \quad \text{ for }k\to\infty\,.
\end{equation*}

We introduce
\begin{equation*}
	B:=\{x\in  J_u\colon v^{t_1}(x)\ne v^{t_2}(x)\}=
	\{x\in  J_u\colon \phi(v^{t_1}(x))\ne \phi(v^{t_2}(x))\}\,.
\end{equation*}
We will prove our claim by contradiction: Assume that $\mathcal{H}^{n-1}(B)>0$. Assume also, without loss of generality, that 
\begin{equation*}
	\tilde B:=\{x\in B \colon v^{t_1}(x)=1, v^{t_2}(x)= -1 \}= \{x\in B \colon \phi(v^{t_1}(x))=\phi(1),\, \phi(v^{t_2}(x))= \phi(-1) \}\,
\end{equation*}
satisfies  $\mathcal{H}^{n-1}(\tilde B)>0$. Let $\delta>0$.
By the properties of sets of finite perimeter under blow-up (see e.g.~\cite[Chapter 5.7]{evans1992measure})
for $\mathcal{H}^{n-1}$ a.e.~$x_0\in \tilde B$, we can find  $\rho_0>0$ such that for every $0<\rho<\rho_0$ it holds:
\begin{equation}\label{ApproxTangent}
\begin{cases}
       |\{u=1\}\cap B_\rho^+(x_0)|\geq \frac{1}{2}|B_\rho(x_0)|-\delta \textrm{ and }  |\{u=1\}\cap B_\rho^-(x_0)|\leq \delta\,, \\[1em]
     1-\delta\leq \dfrac{\mathcal{H}^{n-1}( J_u\cap B_\rho(x_0))}{ \omega(\pi,n-1)\rho^{n-1}}\leq 1+\delta\,,\\
\end{cases}
 \end{equation}  
where $B^{\pm}_\rho(x_0)$ are the half-balls separated by $D_0$, the $(n-1)$-dimensional unit disk  of radius $\rho$, of center $x_0$ and orthogonal to $\nu_u(x_0)$ and  $\omega(\pi,n-1)$ is the measure of the  $n-1$-dimensional unit ball.
By \eqref{StrictConvSlice},
\begin{equation}\label{LocalConv}
\mathcal{H}^{n-1}(\partial_*E_k^{t_i}\cap B_\rho(x_0))\to \mathcal{H}^{n-1}(J_u\cap B_\rho(x_0))\,,\quad \text{ for } i=1,2\,.
\end{equation}
This together with \eqref{ApproxTangent} imply
\begin{equation}\label{ApproxTangentE}
    1-2\delta\leq \dfrac{\mathcal{H}^{n-1}(\partial_*{E_k^{t_i}}\cap B_\rho(x_0))}{ \omega(\pi,n-1)\rho^{n-1}}\leq 1+2\delta\,,\quad \text{ for }i=1,2
\end{equation}
for $k\in\N$ large.
By the Lebesgue Differentiation Theorem,  up to possibly taking $\rho_0$ smaller,  we also have that:
\begin{equation}\label{CloseTo1}
  \dfrac{\mathcal{H}^{n-1}(\{x\in  J_u\cap B_\rho(x_0), v^{t_1}(x)\geq 1-\delta\})}{ \omega(\pi,n-1)\rho^{n-1}}\geq 1-\delta\,,  
\end{equation}
and
\begin{equation*}
    \dfrac{\mathcal{H}^{n-1}(\{x\in  J_u\cap B_\rho(x_0), v^{t_2}(x)\leq\delta -1 \})}{ \omega(\pi,n-1)\rho^{n-1}}\geq 1-\delta\,.
\end{equation*}

Moreover, by the strong convergence \eqref{strong-t1-t2}  we have:
$$\int_{\partial_*E_k^{t_i}}\varphi(x,v_k(x))\dHn\to \int_{J_u}\varphi(x,v^{t_i}(x))\dHn\,,\quad \text{ for } i=1,2$$
for every $\varphi\in C^0_c(\R^n\times \R).$
Hence, by a density argument we obtain
\begin{equation*}
    \mathcal{H}^{n-1}(\{x\in \partial_*E_k^{t_1}\cap B_\rho(x_0), v_k(x)\geq 1-\delta\}) \to \mathcal{H}^{n-1}(\{x\in  J_u\cap B_\rho(x_0), v^{t_1}(x)\geq 1-\delta\})\,,
\end{equation*}
and
\begin{equation*}
    \mathcal{H}^{n-1}(\{x\in \partial_*E_k^{t_2}\cap B_\rho(x_0), v_k(x)\leq \delta -1 \}) \to \mathcal{H}^{n-1}(\{x\in  J_u\cap B_\rho(x_0), v^{t_2}(x)\leq \delta -1 \})\,.
\end{equation*}
Therefore, for $k$ large enough and using \eqref{ApproxTangent},  \eqref{LocalConv} and \eqref{CloseTo1}
it holds
\begin{equation}\label{CloseTo1E}
        \dfrac{ \mathcal{H}^{n-1}(\partial_*E_k^{t_1}\cap B_\rho(x_0)\cap\{v_k(x)\ge 1-\delta\})}{\mathcal{H}^{n-1}(\partial_*E_k^{t_1}\cap B_\rho(x_0)) }
        \ge
1-2\delta\,
\end{equation}
and
\begin{equation}\label{CloseTo0E}
   \dfrac{\mathcal{H}^{n-1}(\partial_*E_k^{t_2}\cap B_\rho(x_0)\cap\{ v_k(x)\leq \delta -1 \})}{\mathcal{H}^{n-1}(\partial_*E_k^{t_2}\cap B_\rho(x_0)) }\geq 1-2\delta\,. 
\end{equation}


By the $L^1_{\rm loc.}$ convergence of  $\chi_{E_k^{t_i}}$ to $(1+u)/2$ for $i=1,2$ and \eqref{ApproxTangent}, for $k$ large, we have:
\begin{equation*}
     |E_k^{t_i}\cap B_\rho^+(x_0)|\geq \frac{1}{2}|B_\rho(x_0)|-2\delta \,,\quad \text{ for }i=1,2\,. 
\end{equation*}
By Fubini's theorem we have:
\begin{equation*}
    |E_k^{t_i}\cap B_\rho^+(x_0)|=\int_{D_0} \int_{ x+\R\nu_u(x_0)}\chi_{E_k^{t_i}\cap B_\rho^+(x_0)}(x,y) \dy \dHn(x)\,,\quad \text{ for }i=1,2\,. 
\end{equation*}
Hence, for $\delta$ chosen sufficiently small, we obtain that
\begin{equation*}
   \H^{n-1}\big( \{x\in D_0\,, x+\R\nu_u(x_0) \textrm{ intersects } E_k^{t_i}\cap B_\rho(x_0) \textrm{ for } i=1,2 \}\big) \geq \frac{5}{6}\omega(\pi,n-1)\rho^{n-1}.
\end{equation*}
In a similar way, the estimate
\begin{equation*}
    |E_k^{t_i}\cap B_\rho^-(x_0)|\leq 2\delta\quad \textrm{ for } i=1,2
\end{equation*}
implies that
\begin{equation*}
      \H^{n-1}\big( \{x\in D_0\,, x+\R\nu_u(x_0) \textrm{ intersects } ({E_k^{t_i}})^c\cap B_\rho(x_0) \textrm{ for } i=1,2 \}\big) \geq \frac{5}{6}\omega(\pi,n-1)\rho^{n-1}\,,
\end{equation*}
where $({E_k^{t_i}})^c$ denotes the complement of ${E_k^{t_i}}$.
Hence,
\begin{equation*}
   \H^{n-1}\big( \{x\in D_0\,, x+\R\nu_u(x_0) \textrm{ intersects } \partial{E_k^{t_i}}\cap B_\rho(x_0) \textrm{ for } i=1,2 \}\big) \geq \frac{2}{3}\omega(\pi,n-1)\rho^{n-1}.
\end{equation*}
Since $U_k$ is a Sobolev function, $U_k$ is absolutely continuous on $x+\R\nu_u(x_0)$ for $\H^{n-1}$ a.e. $x\in D_0$ and we get 
\begin{equation*}
   \H^{n-1}\big( \{x\in D_0\,, x+\R\nu_u(x_0) \textrm{ intersects } U_k^{-1}(t_i)\cap B_\rho(x_0) \textrm{ for } i=1,2 \}\big) \geq \frac{2}{3}\omega(\pi,n-1)\rho^{n-1}.
\end{equation*}
By \eqref{LevelSet} it follows:
\begin{equation}\label{LargeProjection}
   \H^{n-1}\big( \{x\in D_0\,, x+\R\nu_u(x_0) \textrm{ intersects } \partial_*{E_k^{t_i}}\cap B_\rho(x_0) \textrm{ for } i=1,2 \}\big) \geq \frac{2}{3}\omega(\pi,n-1)\rho^{n-1}.
\end{equation}
Let us call $D_1$ the set of points $x'\in D_0$ such that $x'+\R\nu_u(x_0)$ intersects $$\{x\in \partial_*E_k^{t_1}\cap B_\rho(x_0), v_k(x)\geq 1-\delta\}\quad\textrm{ and }\quad\{x\in \partial_*E_k^{t_2}\cap B_\rho(x_0), v_k(x)\leq \delta -1 \}\,.$$ Then, for $k$ large enough, by  \eqref{CloseTo1E}, \eqref{CloseTo0E} and \eqref{LargeProjection}, we obtain that 
$$\H^{n-1}(D_1)\geq \frac{1}{2} \omega(\pi,n-1)\rho^{n-1}.$$
We are ready to prove the contradiction. We have
\begin{equation*}
\begin{split}
     {\Lambda} \geq I_{\eps_k}(u_k,v_k)
 \ge 
		\int_{\{t_1\leq U_k\leq t_2\}} \left(
		\frac{W(v_k)}{\eps_k}+ \frac{\eps_k}{ 2}|\nabla v_k|^2
		\right)
	\frac{W(u_k)}{\e_k}\dx
 \,.
\end{split}
\end{equation*}
Since $ 0<t_1<t_2<\sigma$, we have that
\[
W(\phi^{-1}(t))\geq c(W,t_1,t_2) \text{ for  }t\in[t_1,t_2]\,.
\]
Thus, 
\begin{equation*}
	\begin{split}
		{\Lambda}& \geq \frac{c(t_1,t_2,W)}{\eps_k}\int_{\{t_1\leq U_k\leq  t_2\}} \left(
		\frac{W(v_k)}{\eps_k}+ \frac{\eps_k}{ 2}|\nabla v_k|^2
		\right)\dx\\
		&\geq \frac{c(t_1,t_2,W)}{\eps_k}\int_{\{t_1\leq U_k\leq  t_2\}\cap B_\rho(x_0)} |\nabla (\phi(v_k))|\dx\\
		&\geq \frac{c(t_1,t_2,W)}{\eps_k}\int_{D_1}\int_\R \chi_{\{t_1\leq U_k\leq  t_2\}\cap B_\rho(x_0)}(x'+y \nu_u(x_0)) |\nabla (\phi(v_k))(x'+y\nu_u(x_0))|\dy\d\H^{n-1}(x')\\
  & \geq \frac{c(t_1,t_2,W)}{\eps_k}\int_{D_1}\int_\R \chi_{\{t_1\leq U_k\leq  t_2\}\cap B_\rho(x_0)}(x'+y \nu_u(x_0)) |\partial_y (\phi(v_k))(x'+y\nu_u(x_0))|\dy\d\H^{n-1}(x')\,.
\end{split} 
\end{equation*}
Then, by integrating over $y$ and by definition of $D_1$ we obtain
\begin{equation*}
  \Lambda\geq  \frac{c(t_1,t_2,W)}{\eps_k} \int_{D_1} \phi(1-\delta)-\phi(\delta-1)\dx'.	
\end{equation*}
This last quantity goes to $+\infty$ when $k\to\infty$. Hence, $\H^{n-1}(B)=0$, which proves the claimed equality $v^{t_1}=v^{t_2}$.

\smallskip

\noindent
{\textbf{Step 3: Existence of a convergent  subsequence.}} 
We claim that there exists a strictly monotone sequence $(k_l)_{l\in\N}\subset \N$ and $v\in BV(S;\{-1,1\})$ such that for a.e. $t\in(0,\sigma)$:
\begin{equation}
\label{eq:wstarconv}
	\int_{\partial_*E^t_k} \varphi(x, v_k) \dHn\to \int_{J_u} \varphi(x, v) \dHn \text{ for every }\varphi\in C^0_c(\R^n\times\R)\,.
\end{equation}

\medskip

Indeed,  we apply Lemma \ref{Rear}  with:
\begin{equation*}
    G_k(t):=\int_{\partial_*E^t_k} \bigg( \frac{W(v_k)}{\eps_k}+\frac{\eps_k}{2}|\nabla v_k|^2\bigg) \dHn\,,
\end{equation*}
for every $t\in [0,\sigma]$. Hence, there exists a sequence of functions $(h_k)_{k\in\N}$ mapping $[0,\sigma]$ into itself such that for a.e. $t\in [0,\sigma]$:
\begin{equation}\label{BoundSliceRear}
    \sup_{k\in\N} \int_{\partial_*E^{h_k(t)}_k} \bigg( \frac{W(v_k)}{\eps_k}+\frac{\eps_k}{2}|\nabla v_k|^2\bigg) \dHn <+\infty\,.
\end{equation}

Moreover, for a.e. $t\in[0,\sigma]$, there exist $0<s_1\leq h_k(t)\leq s_2<\sigma$ $\forall k\in\N$, such that $\chi_{E^{s_1}_k}$ and $\chi_{E^{s_2}_k}$ converge to ${\frac{u+1}{2}}$ in $L^1_{\rm loc}(\R^n)$ when $k\to+\infty$. Since $\chi_{E^{s_2}_k}\le \chi_{E^{h_k(t)}_k}\le \chi_{E^{s_1}_k}$,  $\chi_{E^{h_k(t)}_k}\to (1+u)/2$ in $L^1_{\rm loc}(\R^n)$. By lower semicontinuity of the perimeter and the fact that 
\begin{equation*}
     \int_0^\sigma {\rm Per}(E^{h_k(t)}_k)\dt= \int_0^\sigma {\rm Per}(E^t_k)\dt \to  \sigma\H^{n-1}(J_u)\,,
\end{equation*}
we have that $E^{h_k(t)}_k$ converges strictly in $BV_{\rm loc}(\R^n)$ to ${\frac{u+1}{2}}$ for a.e. $t\in[0,\sigma]$.

Next we will apply Lemma  \ref{ConvRear}. Using the notation from the hypothesis of that Lemma, we set $X=\mathcal M(\R^{n+1})$, and as convergence ``$\to$'' on $X$ is the we take the weak-* convergence on $\mathcal M(\R^{n+1})$. The map $F_k:[0,\sigma]\to X=\mathcal M(\R^{n+1})$ is defined by 
\begin{equation*}
   \int_{\R^n\times\R} \varphi\,\d (F_k(t))= \int_{\partial_*E^{h_{k}(t)}_{k}} \varphi(x, v_{k}) \dHn \text{ for every }\varphi\in C^0_c(\R^n\times\R)\,.
\end{equation*}

By Lemma \ref{lem:Hdefslicing}, \eqref{BoundSliceRear} and  the strict convergence of $E^{h_k(t)}_{k_t}$   in $BV_{\rm loc}(\R^n)$ to $(1+u)/2$, we can apply  Theorem \ref{thm1OR} (\cite[Theorem 1]{OR}) to the sequence $(F_k(t))_{k\in\N}$ for a.e. $t\in [0,\sigma]$ and for any subsequence. Hence, $(F_k)_{k\in\N}$ satisfies the first assumption of Lemma \ref{ConvRear}. By \textbf{Step 2}, it also satisfies the second assumption. 

Therefore, by Lemma \ref{ConvRear}, there exist $v\in BV(S;\{-1,1\})$ and  a strictly monotone sequence $(k_l)_{l\in\N}\subset \N$ such that for a.e. $t\in[0,\sigma]$ :
\begin{equation*}
	\int_{\partial_*E^{h_{k_l}(t)}_{k_l}} \varphi(x, v_{k_l}) \dHn\to \int_{J_u} \varphi(x, v) \dHn \text{ for every }\varphi\in C^0_c(\R^n\times\R)\,.
\end{equation*}

In order to prove \ref{eq:wstarconv},  it remains to get rid of the rearrangement $h_k$.
To do so we observe that by the third property of $(h_k)_{k\in\N}$ in Lemma \ref{Rear}, we have that for every $\varphi\in C^0_c(\R^n\times\R)$, and for every $l\in\N$, the quantity
\begin{equation*}
   I_1^l:= \int_0^\sigma \bigg|\int_{\partial_*E^{h_{k_l}(t)}_{k_l}} \varphi(x, v_{k_l}) \dHn- \int_{J_u} \varphi(x, v) \dHn\bigg| \dt
    \end{equation*}
    is equal to
    \begin{equation*}
   I_2^l:= \int_0^\sigma \bigg|\int_{\partial_*E^{t}_{k_l}} \varphi(x, v_{k_l}) \dHn- \int_{J_u} \varphi(x, v) \dHn\bigg| \dt\,.
\end{equation*}
But for a.e. $t\in[0,\sigma]$, 
\begin{equation*}
    \bigg|\int_{\partial_*E^{t}_{k_l}} \varphi(x, v_{k_l}) \dHn- \int_{J_u} \varphi(x, v) \dHn\bigg|\leq \|\varphi\|_{L^\infty}\big({\rm Per}(E^t_{{k_{l}}})+\H^{n-1}(J_u)\big)\,,
 \end{equation*}    
     and \begin{equation*}\int_0^\sigma\|\varphi\|_{L^\infty}{\rm Per}(E^t_{k_l}) \to \sigma\|\varphi\|_{L^\infty} \H^{n-1}(J_u)\,.
\end{equation*}
Hence, by the dominated convergence theorem, $\lim\limits_{l\to+\infty}I_2^l=\lim\limits_{l\to+\infty}I_1^l=0$. Thus we have shown that for all $\varphi\in C^0_c(\R^n\times \R)$, the sequence of functions
\[
t\mapsto \int_{\partial_* E_{k_l}^t}\varphi(x,v_{k_l}(x))\d\H^{n-1}(x)
\]
converges in $L^1(0,\sigma)$ for $l\to\infty$ to the constant function 
\[
t\mapsto \int_{J_u}\varphi(x,v_{k_l}(x))\d\H^{n-1}(x)\,.
\]
Taking further subsequences that converge for almost every $t\in (0,\sigma)$ for a dense countable subset of functions $\varphi\in C^0_c(\R^n\times \R)$ and extracting a diagonal sequence yields the subsequence fulfilling \ref{eq:comp_str_mfp2}.

\smallskip

\noindent
{\textbf{Step 4: Measure-function pair convergence.}}
Finally, we prove that the sequence from the previous step also fulfills \ref{eq:comp_str_mfp}. 
    We first have to prove that the strictly monotone sequence $(k_l)_{l\in\N}\subset \N$ and $v\in BV(S;\{-1,1\})$ obtained in \textbf{Step 3} satisfy for every $\varphi\in C^0_c(\R^n\times \R)$:
    \begin{equation*}
        \lim_{l\to\infty}\int_{\R^n}\varphi(x,v_{k_l}(x))|\nabla U_{k_l}|\dx=\sigma\int_{J_u}\varphi(x,v(x))\dHn\,.    \end{equation*}
    Indeed, by the coarea formula, we have that
    \begin{equation*}
        \int_{\R^n}\varphi(x,v_{k_l}(x))|\nabla U_{k_l}|\dx=\int_\R \int_{\partial_*E_{{k_l}}^t} \varphi(x,v_{k_l}(x)) \dHn \dt\,.
    \end{equation*}
    For a.e. $t\in(0,\sigma)$,
    \begin{equation*}
        \int_{\partial_*E_{{k_l}}^t} \varphi(x,v_{k_l}(x)) \dHn\to \int_{J_u} \varphi(x,v(x)) \dHn\,,
    \end{equation*}
    and for a.e. $t\notin(0,\sigma)$,
    \begin{equation*}
        \int_{\partial_*E_{{k_l}}^t} \varphi(x,v_{k_l}(x)) \dHn\to 0\,.
    \end{equation*}
    Moreover,
    \begin{equation*}
         \int_{\partial_*E_{{k_l}}^t} \varphi(x,v_{k_l}(x)) \dHn\leq \|\varphi\|_{L^\infty}{\rm Per}(E_{{k_l}}^t)\,,
    \end{equation*}
    for a.e. $t\in \R$
    \begin{equation*}
      \lim_{l\to\infty}  \|\varphi\|_{L^\infty}{\rm Per}(E_{{k_l}}^t)=\|\varphi\|_{L^\infty} \mathcal{H}^{n-1}(J_u)\chi_{(0,1)}(t)
    \end{equation*}
    and
    \begin{equation*}
        \lim_{l\to\infty} \int_\R \|\varphi\|_{L^\infty}\Per(E_{{k_l}}^t) \dt= \|\varphi\|_{L^\infty} \mathcal{H}^{n-1}(J_u).
    \end{equation*}
    Hence, by the dominated convergence theorem,
    \begin{equation*}
        \lim_{l\to\infty}\int_{\R^n}\varphi(x,v_{k_l}(x))|\nabla U_{k_l}|\dx= \sigma \int_{J_u}\varphi(x,v(x))\dHn\,.
    \end{equation*}
It remains  to show $\lim_{j\to\infty}\int_{\S_{k_l,j}}|v_{k_l}|^q|\nabla U_{k_l}|\d x=0$ uniformly in $l$, where  $\S_{k_l,j}:=\{x\in\R^n:|x|\geq j \text{ or } |v_{k_l}(x)|\geq j\}$. 
Since $|\nabla U_{k_l}|\L^n\wsto \sigma \H^{n-1}\res J_u$, we have by Prokhorov's theorem (see \cite[Theorem 8.6.2]{bogachev2007measure})
\[
\lim_{j\to\infty}\int_{\{x:|x|\geq j\}}|\nabla U_{k_l}|\d x= 0\quad\text{ uniformly in } l\in\N\,.
\]
Hence (supposing $\e_{k_l}\in (0,1)$) 
\[
  \begin{split}
    \limsup_{j\to\infty}&\int_{\{x:|x|\geq j\}}|v_{k_l}|^q|\nabla U_{k_l}|\d x\\
&\leq \limsup_{j\to\infty} \left(\int_{\{x:|x|\geq j\}}|v_{k_l}|^p|\nabla U_{k_l}|\d x\right)^{q/p}\left(\int_{\{x:|x|\geq j\}}|\nabla U_{k_l}|\d x\right)^{1-q/p}\\
&\leq C\limsup_{j\to\infty} \left(\int_{\{x:|x|\geq j\}}\frac{W(v_{k_l})}{\e_{k_l}}\left(\frac{\e_{k_l}}{2}|\nabla u_{k_l}|^2+\frac{W(u_{k_l})}{\e_{k_l}}\right)\d x\right)^{q/p}\left(\int_{\{x:|x|\geq j\}}|\nabla U_{k_l}|\d x\right)^{1-q/p}\\
&\leq C\limsup_{j\to\infty}\Lambda^{q/p}\left(\int_{\{x:|x|\geq j\}}|\nabla U_{k_l}|\d x\right)^{1-q/p}\\
&=0 \quad\text{ uniformly in } l\in\N\,.
  \end{split}
\]
Furthermore, using again the same estimates,
\[
  \begin{split}
    \limsup_{j\to\infty}\int_{\{x:v_{k_l}(x)\geq j\}}|v_{k_l}|^q|\nabla U_{k_l}|\d x
&\leq \limsup_{j\to\infty}j^{q-p} \int_{\{x:v_{k_l}(x)\geq j\}}|v_{k_l}|^p|\nabla U_{k_l}|\d x\\
&\leq \limsup_{j\to\infty}j^{q-p} \Lambda\,.
\end{split}
\]
This completes the proof of \eqref{eq:comp_str_mfp}.
\end{proof}
%

%
We next prove the $\liminf$ inequality.
\begin{proposition}[Lower bound]\label{prop:lower_bound}
Let 
$((u_\eps,v_\eps))_\e\subset (L^1_\loc(\R^n))^2$ be a sequence that converges to $(u,v)\in BV_\loc(\R^n;\{-1,1\})\times BV(S;\{-1,1\})$, with $S=\llbracket J_u,\star\nu_u,1\rrbracket$ in the following sense: $u_\eps\to u$ strictly in $BV_\loc(\R^n)$ and for a.e. $t\in[0,\sigma]$:
  \begin{equation}\label{lower_bound_slice}
        \int_{\partial_* E_\eps^t} \varphi(x,v_\eps)\dHn\to \int_{J_u} \varphi(x,v)\dHn \text{ for every }\varphi\in C^0_c(\R^n\times\R)\,,
    \end{equation}
    where $E_\eps^t=\{\phi\circ u_\eps > t\}$. Then 
\begin{equation*}
	\liminf_{\eps\searrow0}I_\eps(u_\eps,v_\eps)\ge I(u,v)\,.
\end{equation*}
\end{proposition}

\begin{proof} 
We use the same notation as in the proof of Proposition \ref{prop:compactness}. Combining the Modica-Mortola trick with the coarea formula,
\begin{equation*}
\begin{split}
      I_\eps(u_\eps,v_\eps) &=  \int_{\R^n} \left(
	\frac{W(v_\eps)}\eps+ \frac\eps{ 2}|\nabla v_\eps|^2
	\right)
	\left(
	\frac{W(u_\eps)}\eps+ \frac\eps{ 2}|\nabla u_\eps|^2
	\right)\dx \\
 &\geq \int_{\R^n}|\nabla U_\e|\left(
	\frac{W(v_\eps)}\eps+ \frac\eps{ 2}|\nabla v_\eps|^2
	\right)\dx\\
&\geq \int_{0}^{\sigma}\int_{\partial_*E_\e^t}\left(\frac{W(v_\eps)}\eps+ \frac\eps{ 2}|\nabla v_\eps|^2
	\right)\d\H^{n-1}\,.
\end{split}  
\end{equation*}
By Fatou's lemma, 
\[
    \liminf_{\e\to 0}\int_{0}^{\sigma}\int_{\partial_*E_\e^t}\left(\frac{W(v_\eps)}\eps+ \frac\eps{ 2}|\nabla v_\eps|^2
    \right)\d\H^{n-1}\ge\int_{0}^{\sigma}  \liminf_{\e\to 0}\int_{\partial_*E_\e^t}\left(\frac{W(v_\eps)}\eps+ \frac\eps{ 2}|\nabla v_\eps|^2
    \right)\d\H^{n-1}\,.
\]

The assumed convergence implies in particular that  
\[
    \chi_{E_\e^t} \to  \chi_{\{u=1\}} \quad\text{ strictly in } BV_\loc(\R^n)
\]
for  every $t\in (0,\sigma)\setminus T$, where $T$ is a null set. For such a $t$, if \[\liminf_{\eps\to 0} \int_{\partial_* E_\eps^{t}} \left(
        \frac{W(v_\eps)}\eps+ \frac\eps{ 2}|\nabla v_\eps|^2
        \right) \dHn<+\infty \]
        we obtain by Theorem \ref{thm1OR}
        the existence of $v^t\in BV(S;\{-1,1\})$ such that
    \begin{equation*}
         \liminf_{\eps\to 0} \int_{\partial_* E_\eps^{t}} \left(
        \frac{W(v_\eps)}\eps+ \frac\eps{ 2}|\nabla v_\eps|^2
        \right) \dHn\geq \sigma \H^{n-2}(J_{v^t})\,,
    \end{equation*}
and
  \begin{equation*}
        \int_{\partial_* E_\eps^t} \varphi(x,v_\eps)\dHn\to \int_{J_u} \varphi(x,v^t)\dHn \text{ for every }\varphi\in C^0_c(\R^n\times\R)\,.
    \end{equation*}  
By assumption \eqref{lower_bound_slice}, for a.e. $t\in [0,\sigma]$ we have $v=v^t$.
Moreover, if \[\liminf_{\eps\to 0} \int_{\partial_* E_\eps^{t}} \left(
        \frac{W(v_\eps)}\eps+ \frac\eps{ 2}|\nabla v_\eps|^2
        \right) \dHn=+\infty\,, \]
 then $\sigma \H^{n-2}(J_v)\le +\infty$.
 
 Hence,
\[
  \begin{split}
    \liminf_{\e\to 0}\int_{0}^{\sigma}\int_{\partial_*E_\e^t}\left(\frac{W(v_\eps)}\eps+ \frac\eps{ 2}|\nabla v_\eps|^2
    \right)\d\H^{n-1}
&\geq
    \int_{0}^{\sigma} \liminf_{\e\to 0} \int_{\partial_*E_\e^t}\left(\frac{W(v_\eps)}\eps+ \frac\eps{ 2}|\nabla v_\eps|^2
    \right)\d\H^{n-1}\\
    &\geq  \sigma^2 \H^{n-2}(J_v)\,,
  \end{split}
\]
which proves the claim.
\end{proof}

\section{Second main result}
\label{sec:proofthm2}
In this section we state and prove our second main result. 
We stress that from now on we restrict the analysis to functionals corresponding to the specific choice $n=3$ and
\begin{equation}\label{W}
 W(t)=   \bar W(t):=(1-t^2)^2\,.
\end{equation}
As in the previous section  we use the  notation 
\[
  \begin{split}
    \phi(s)&=\int_{-1}^s \sqrt{2\bar W(t)}\d t\\
\sigma&=\phi(1)\\
    U_\e&=\phi\circ u_\e\,.
  \end{split}
\]
We introduce the function $a^{\bar\omega}\colon\R\to\R$ given by
\begin{equation}\label{def:a}
	a^{\bar\omega}(t)\defas \bar\om(t)a_1+ (1-\bar\omega(t))a_2 \,,
\end{equation}
with $a_1,a_2>0$, $\bar\om\in C^\infty_c(\R)$, $0\le \bar\omega\le1$, $\bar\omega({ -1})=0$ and $\bar\omega(1)=1$. For later convenience, we observe that 
\begin{equation}\label{bound_a}
\min\{a_1,a_2\}\le 	a^{\bar\omega}(t)\le \max\{a_1,a_2\}\quad \forall t\in \R\,.
\end{equation}
For $\eps>0$ we consider the family of functionals $J_\eps\colon(L^1(\R^3))^2\to[0,+\infty]$ given by
\begin{equation}\label{def:J_e}
J_\eps(u,v):= \int_{\R^3} \frac{1}{\eps
}a^{\bar\omega}(v)\left(\frac{\bar W'(u)}{\eps}-\eps\Delta u\right)^2\dx\,,
\end{equation}
if $(u,v)\in W^{2,2}_\loc(\R^3)\times L^1_{\loc}(\R^3)$,  and $+\infty$ otherwise.
Our second main result is the following.

     \begin{theorem}[Lower bound and compactness of $I_\eps+J_\eps$]\label{thm2} Let $n=3$ and $\bar W$ be as in \eqref{W}. 
  	Let $I_\eps, J_\eps, M_\eps$ be defined as in \eqref{def:I_e}, \eqref{def:J_e} and \eqref{def:M_e} respectively, where we assume $W(t)\equiv \bar W(t)$. Then the following hold:
  	\begin{enumerate}[label=$(\arabic*)$]
  		\item Compactness.	Let $((u_\eps,v_\eps))_\e\subset (L^1_\loc(\R^3))^2$ be a sequence such that  
  		
  		$$\sup( I_\eps(u_\eps,v_\eps)+M_\eps(u_\eps))<+\infty\,,$$ 
  		\begin{equation}\label{assumption_J}
  			J_\eps(u_\eps,v_\eps)<{\sigma}(8\pi\min\{a_1,a_2\}-\delta)\,,
  		\end{equation}
  		for some $\delta>0$. Then there exists  $(u,v)\in BV_\loc(\R^3;\{{ -1},1\})\times BV(S;\{{ -1},1\})$, with $S=\llbracket J_u,\star\nu_u,1\rrbracket$, such that, up to subsequence,  
\begin{equation}\label{eq:12}
  \begin{split}
u_\e&\to u\quad\text{ strictly in }BV\\
    \left(\left(\frac{\e}{2}|\nabla u_\e|^2+\e^{-1}\bar W(u_\e)\right)\L^3,v_\e\right)&\to (\sigma \H^{2}\res J_u, v)\quad\text{ in }L^q
  \end{split}
\end{equation}
for every $q\in [1,4)$ as measure-function pairs.
  		\item Lower bound. Let $((u_\eps,v_\eps))_\e\subset (L^1_\loc(\R^3))^2$ be a sequence and  $(u,v)\in BV_\loc(\R^3;\{{ -1},1\})\times BV(S;\{{ -1},1\})$ with $S=\llbracket J_u,*\nu_u,1\rrbracket$, such that \eqref{eq:12} holds. Then 
  		\begin{equation}\label{liminf-inequality-22}
  			\liminf_{\eps\searrow0}(I_\eps(u_\eps,v_\eps)+ J_\eps(u_\eps,v_\eps))\ge I(u,v)+ J(u,v)\,,
  		\end{equation}
  		where $I$ is as in \eqref{def:I} and $J\colon (L^1_\loc(\R^3))^2\to[0,+\infty]$ is defined as
  		\begin{equation}\label{def:J}
  			J(u,v)\defas \sigma \int_{J_u}\left(a_1\frac{1+v}{2}+a_2\frac{1-v}{2}\right)|H_{J_u}|^2\d\mathcal{H}^{2}\,,
  		\end{equation}
  		if $(u,v)\in BV_\loc(\R^3;\{{ -1},1\})\times BV(S;\{{ -1},1\})$, and is equal to $+\infty$ otherwise.
  	\end{enumerate}
  \end{theorem}   
%
\subsection{Proof of Theorem \ref{thm2}}
%
In the proof of our second theorem, we are going to rely heavily on the analysis from \cite{RS}. In the following, we summarize the results from that reference that will be useful for our purpose. 
 \begin{notation}
For $u_\e\in W^{1,2}_{\mathrm{loc}}(\R^n)$ let 
\begin{equation}\label{mu}
     \begin{split}
    \mu_\e&\defas\left(\frac{\e}{2} |\nabla u_\e|^2+\e^{-1}\bar W(u_\e)\right)\L^n\\
    \xi_\e&\defas\left(\frac{\e}{2} |\nabla u_\e|^2-\e^{-1}\bar W(u_\e)\right)\L^n
  \end{split}
\end{equation}
We omit the dependence of $\mu_\e,\xi_\e$ from $u_\e$ from the notation. No confusion will arise from this. 
 \end{notation}

In all of the remaining statements of the current subsection, it will be assumed 
\[
 \begin{split}
   \mu_\e&\wsto  \mu\quad\text{  in }\mathcal M(\R^3)\,,\\
 \sup_{\e>0}\int_{\R^3}\frac1\e\left(\e^{-1}\bar W'(u_\e)-\e\Delta u_\e\right)^2\d x&<+\infty\,.  
 \end{split}
\]
\begin{proposition}[{\cite[Proposition 4.3]{tonegawa2002phase} and \cite[Proposition 4.9]{RS}}]
\label{prop:discrepancy}
Under the above assumptions, $|\xi_\e|\wsto 0$ in $\mathcal M(\R^3)$.
\end{proposition}
In the statement of the following proposition  we identify $G^o(3,2)$ with the set of simple unit elements of $\Lambda_{2}(\R^3)$. 
\begin{proposition}[{\cite[Proposition 4.10]{RS}}]
\label{prop:1stvariationVe}
Let $V_\e\in \mathcal M_+(\R^3\times G^o(3,2))$ be the oriented varifold defined by 
\[
V_\e(\varphi)=\int_{\R^3}\varphi(x,\star \nabla u_\e(x)/|\nabla u_\e(x)|)\d\mu_\e\,.
\]
Then  the first variation of $V_\e$ is given by 
\[
  \begin{split}
    (\delta V_\e)(\varphi)&=-\int (-\e\Delta u_\e+\e^{-1}\bar W'(u_\e))\nabla u_\e(x)\cdot\varphi\d x\\
   &\quad +\int \nabla \varphi: \frac{\nabla u_\e}{|\nabla u_\e|}\otimes \frac{\nabla u_\e}{|\nabla u_\e|} \d\xi_\e \quad \text{ for }\varphi\in C^1_c(\R^3;\R^3)\,.
  \end{split}
\]
\end{proposition}


\begin{theorem}{\cite[Theorems 4.1 and 5.1]{RS}}
\label{thm:lowerbd_integrality}
There exists a rectifiable  $2$-varifold  $V$ with the following properties:
\begin{itemize}
\item[(i)] There exists a subsequence (no relabeling) such that $V_\e\wsto V$ in  $\mathcal M(\R^n\times G^o(3,2))$. In particular $\|V\|=\mu$.
\item[(ii)] The varifold $V$ has generalized  mean curvature $H_V\in L^2_{\|V\|,\loc}(\R^3;\R^3)$ satisfying
\[
\int |H_V|^2\d\|V\|\leq \liminf_{\e\to 0}\e^{-1}\int_{\R^n}(-\e\Delta u_\e+\e^{-1}\bar W(u_\e))^2\d x\,.
\]
\item[(iii)] The varifold $ \sigma^{-1}V$ is integral.
\end{itemize}

\end{theorem}

As in \cite{OR}, we are going to use the  Li-Yau inequality \cite{li1982new,KuSc04} to obtain the crucial property that the limit surface is of density one. In order to state the Li-Yau inequality, suppose that $\tilde V\in \sfRV_2^o(\R^3)$. The two-dimensional density of $\|\tilde V\|$ at $x\in \R^3$ is defined by 
\begin{equation*}
    \theta_2(x,\|\tilde V\|)=\lim_{r\to 0} \dfrac{\|\tilde V\|(B_r(x))}{\omega(2,r)}\,,
\end{equation*}
with $\omega(2,r)$ the volume of the  ball of radius $r$ in $\R^2$.
This limit exists $\|\tilde V\|$-almost everywhere. The Li-Yau inequality states that if $\tilde V$ possesses a mean curvature vector $H_{\tilde V}\in L^2_{\|\tilde V\|}(\R^3)$ then 
\begin{equation*}
 \theta_2\leq \frac1{4\pi}\int |H_{\tilde V}|^2\d\|\tilde V\|\,.
\end{equation*}
In particular, if $\tilde V\in \sfIV_2^o(\R^3)$, the inequality implies 
  \begin{equation}\label{eq:7}
    \int |H_{\tilde V}|^2\d\|\tilde V\|<8\pi\quad \Rightarrow \quad \theta_2=1 \quad \|\tilde V\|-\text{almost everywhere.}
  \end{equation}

 We divide the proof of Theorem \ref{thm2} into several steps. 
 The next proposition is the \emph{compactness} part of that Theorem.

\begin{proposition}[Compactness]\label{prop:compactness2}  Let $n=3$. 
Let $I_\eps, J_\eps, M_\eps$ be defined as in \eqref{def:I_e}, \eqref{def:J_e} and \eqref{def:M_e} respectively. 
  		Let $((u_\eps,v_\eps))_\e\subset (L^1_{\loc}(\R^3))^2$ be a sequence such that  
  		
  		$$\sup( I_\eps(u_\eps,v_\eps)+M_\eps(u_\eps))<+\infty\,,$$ 
  		\begin{equation}\label{assumption_J_2}
  			J_\eps(u_\eps,v_\eps)<\sigma(8\pi\min\{a_1,a_2\}-\delta)\,,
  		\end{equation}
  		for some $\delta>0$. Then there exists  $(u,v)\in BV(\R^3;\{-1,1\})\times BV(S;\{-1,1\})$, with $S=\llbracket J_u,*\nu_u,1\rrbracket$, such that, up to a subsequence,  
    \begin{equation*}
        u_\eps\to u \quad \text{ strictly in } BV_{\rm loc}(\R^3)\,,
    \end{equation*}
    \begin{equation}\label{eq:82}
    \left(\left(\frac{\e}{2}|\nabla u_\e|^2+\e^{-1}W(u_\e)\right)\L^3,v_\e\right)\to (\sigma \H^{2}\res J_u, v)\quad\text{ in }L^q
  \end{equation}
   as measure-function pairs for every $q\in [1,4)$.
   Additionally we have the measure-function pair convergence 
   \begin{equation}\label{eq:mfpnew}
    \left(|\nabla (\phi\circ u_\e)|\L^3,v_\e\right)\to (\sigma \H^{2}\res J_u, v)\quad\text{ in }L^q\,,
  \end{equation}
  again for $q\in [1,4)$.
\end{proposition}

\begin{proof} We will show the convergence \eqref{eq:mfpnew} using  Proposition \ref{prop:compactness}. Indeed, \eqref{eq:mfpnew} follows from that proposition if there exists $u\in BV(\R^3;\{-1,1\})$ such that, up to subsequence, $u_\eps\to u$ strictly in $BV(\R^3)$.
  Taking a further subsequence, we may assume that 
 $u_\eps\to u$ a.e. on $\R^3$. We define the set of finite perimeter $E\subset \R^3$ by writing $u(x)=2\chi_E(x)-1$. 
 From the Modica-Mortola trick, we have that 
 \begin{equation}\label{eq:6}
  \int_{\R^3} |\nabla (\phi\circ u_\eps)| \dx
 \end{equation}
 is uniformly bounded. 

\medskip

We have that $\phi(t)=\int_{-1}^t\sqrt{2}|1-s^2|\d s$, and hence $|\phi(t)|\leq C (|t|^3+1)$. Let $K\subset \R^3$ be compact. We may assume $\e\in (0,1)$. By $\int_{K}|u_\e|^4\leq C(K) (\int_K \bar W(u_\e)\d x+1)\leq C(K)(\Lambda+1)$, we get that $\|u_\e\|_{L^q(K)}\leq C(K,\Lambda)$ for every $q\leq 4$. In particular, $\|\phi\circ u_\e\|_{L^1(K)}\leq C(K,\Lambda)$, independently of $\e$. Combining this with \eqref{eq:6}  and the
 compactness theorem for  $BV$ functions, we can find a subsequence and $U\in BV(\R^3)$ such that $\phi\circ u_\eps\to U$ in $L^1_{\loc}(\R^3)$ and almost everywhere when $\eps\to 0$. Hence, for a.e. $x\in\R^3$, $\phi\circ u(x)=U(x)$ and $u\in BV(\R^3)$. Thus, up to a subsequence, $u_\eps\stackrel{*}{\rightharpoonup}u$ in $BV(\R^3)$.

\medskip
 
Let $\mu_\eps$ be as in \eqref{mu}
and let $\nu_\eps\colon\R^3\to\partial B_1(0)$ be a Borel-measurable function extending $\nabla u_\eps/|\nabla u_\eps|$ on $\{\nabla u_\eps=0\}$.
We define $V_\eps:=\mu_\eps\otimes\nu_\eps$ to be the corresponding generalized varifold, that is \begin{equation*}
	\int_{\R^3\times G^o(3,2)}\varphi(x,S)\d V_\eps(x,S)
	= \int_{\R^3}\varphi(x,\star\nu_\eps(x))\d\mu_\eps(x)\quad\text{ for }\varphi\in C^0_c(\R^{3}\times\R)\,,
\end{equation*}
where we identify $S\in G^o(3,2)$ with the simple unit vector in $\Lambda(3,2)$ orienting $S$. 
Then as $\|V_\eps\|=\mu_\eps$ and $\mu_\e\wsto \mu\in \mathcal M_+(\R^3)$, we obtain by Theorem \ref{thm:lowerbd_integrality}  that there exists $V\in \mathcal M_+(\R^3\times G^o(3,2))$ possessing generalized mean curvature $H_V\in L_{\|V\|}^2(\R^3)$ such that  
\[
\begin{split}
V_\eps&\wsto V\quad \text{ in }\mathcal M_+(\R^3\times G^o(3,2))  \,,\\
\sigma^{-1} \|V\|&= \H^{2}\res J_u\,,\\
\int_{\R^3}|H_V|^2\d\|V\|&\le  \liminf_{\eps\searrow0} 	\int_{\R^3} \frac{1}{\eps}\left(\frac{\bar W'(u_\eps)}{\eps}-\eps\Delta u_\eps\right)^2\dx\,.
	\end{split}
\]


We write $\tilde V=\sigma^{-1}V$.
Since $H_{\tilde V}=H_V$, we obtain 
\begin{equation*}
\int_{\R^3}|H_{\tilde V}|^2\d\|\tilde V\|\leq 8\pi-\tilde \delta\,,
\end{equation*}
with $\tilde \delta=\frac{\delta}{\sigma}$. By  \eqref{eq:7}, $\tilde V$ is of density one. Therefore $S:=\underline c(\tilde V)=\llbracket J_u,\star\nu_u,1\rrbracket$ and from \cite[Lemma 1]{OR}  we have that 
\begin{equation*}
\int_{\R^3}\left(\frac {\bar W(u_\eps)}\eps+\frac\eps2|\nabla u_\eps|^2\right)\dx=M(\underline c(V_\eps))\to M(\underline c(V))=\sigma\H^{2}(J_u)=\sigma|Du|(\R^3)\,.
\end{equation*}
By the weak convergence $u_\eps\stackrel{*}{\rightharpoonup}u$ in $BV(\R^3)$ we have 
\begin{equation*}
\begin{split}
\sigma|Du|(\R^3)& \le \limsup_{\eps\to 0} \sigma|D u_\eps|(\R^3)\\ &=
 \limsup_{\eps\to 0} |D(\phi(u_\eps))|(\R^3)\\ &= \limsup_{\eps\to 0}\int_{\R^3}\sqrt{2\bar W(u_\eps)}|\nabla u_\eps|\dx \\
&\le 
\limsup_{\eps\to 0} \int_{\R^3}\left(\frac {\bar W(u_\eps)}\eps+\frac\eps2|\nabla u_\eps|^2\right)\dx=\sigma|Du|(\R^3)\,
\end{split}
\end{equation*}
which in turn implies $|Du_\eps|(\R^3)\to |Du|(\R^3)$. Hence we infer that $u_\eps\to u$ strictly in $BV$. This proves \eqref{eq:mfpnew}.

\medskip

To show \eqref{eq:82}, we can observe that for every $\varphi\in C^0_c(\R^3\times\R)$,
\begin{equation*}
\begin{split}
     \bigg|  \int_{\R^3} \varphi(x,v_\eps)\d\mu_\eps- \int_{\R^3} \varphi(x,v_\eps)|\nabla U_\eps|\d x \bigg|&\le \|\varphi\|_{L^\infty(\R^3)}\left(\mu_\eps(\R^3)- \int_{\R^3}|\nabla U_\eps|\d x \right)\\
     &\le \|\varphi\|_{L^\infty(\R^3)} |\xi_\eps|(\R^3)\,,
\end{split}
\end{equation*}
where we have used $(a^2+b^2)-2ab=(a-b)^2\leq |a^2-b^2|$ with $a=\sqrt{\e}|\nabla u_\e|$ and $b=\sqrt{2W(u_\e)/\e}$.
The right hand side above goes to $0$ thanks to Proposition \ref{prop:discrepancy}.

It remains  to show $\lim\limits_{j\to\infty}\int_{\S_{\eps,j}}|v_{\eps}|^q\d\mu_\eps=0$ uniformly in $l$, where  $\S_{\eps,j}:=\{x\in\R^n:|x|\geq j \text{ or } |v_{\eps}(x)|\geq j\}$. 
As in the \textbf{Step 4} of the proof of Proposition \ref{prop:compactness}, 
since $\mu_\eps\wsto \sigma \H^{n-1}\res J_u$, we have by Prokhorov's theorem
\[
\lim_{j\to\infty}\mu_\eps(\{x:|x|\geq j\})= 0\quad\text{ uniformly in } l\in\N\,.
\]
Hence 
\[
  \begin{split}
    \limsup_{j\to\infty}&\int_{\{x:|x|\geq j\}}|v_{\eps}|^q\d\mu_\eps\\
&\leq \limsup_{j\to\infty} \left(\int_{\{x:|x|\geq j\}}|v_{\eps}|^p\d\mu_\eps\right)^{q/p}\left(\int_{\{x:|x|\geq j\}}\d\mu_\eps\right)^{1-q/p}\\
&\leq C\limsup_{j\to\infty} \left(\int_{\{x:|x|\geq j\}}\e^{-1}\Bar{W}(v_{\eps})\left(\frac{\e}{2}|\nabla u_{\eps}|^2+\e^{-1}\Bar{W}(u_{\eps})\right)\d x\right)^{q/p}\left(\int_{\{x:|x|\geq j\}}\d\mu_\eps\right)^{1-q/p}\\
&\leq C\limsup_{j\to\infty}\Lambda^{q/p}\left(\int_{\{x:|x|\geq j\}}\d\mu_\eps\right)^{1-q/p}\\
&=0 \quad\text{ uniformly in } l\in\N\,.
  \end{split}
\]
Furthermore, using again the same estimates,
\[
  \begin{split}
    \limsup_{j\to\infty}\int_{\{x:v_{\eps}(x)\geq j\}}|v_{\eps}|^q\d\mu_\eps
&\leq \limsup_{j\to\infty}j^{q-p} \int_{\{x:v_{\eps}(x)\geq j\}}|v_{\eps}|^p\d\mu_\eps\\
&\leq \limsup_{j\to\infty}j^{q-p} \Lambda\,.
\end{split}
\]
This completes the proof.
\end{proof}
\begin{proposition}[Lower bound]\label{prop:lower-bound2}  Let $n=3$.
    	Let $I_\eps, J_\eps, M_\eps$ be defined as in \eqref{def:I_e}, \eqref{def:J_e} and \eqref{def:M_e} respectively. 
  	 Let $((u_\eps,v_\eps))_\e\subset (L^1_{\loc}(\R^3))^2$ be a sequence that converges to $(u,v)\in BV(\R^3;\{-1,1\})\times BV(S;\{-1,1\})$, with $S=\llbracket J_u,*\nu_u,1\rrbracket$, in the following sense: 
     \begin{equation*}
        u_\eps\to u \quad \text{ strictly in } BV_{\rm loc.}(\R^3)\,,
    \end{equation*}
    \begin{equation}\label{eq:8}
    \left(\left(\frac{\e}{2}|\nabla u_\e|^2+\e^{-1}\bar W(u_\e)\right)\L^3,v_\e\right)\to (\sigma \H^{2}\res J_u, v)\quad\text{ in }L^q
  \end{equation}
as measure-function pairs for every $q\in [1,4)$.
    Then 
  		\begin{equation}\label{liminf-inequality-2}
  			\liminf_{\eps\searrow0}I_\eps(u_\eps,v_\eps)+ J_\eps(u_\eps,v_\eps)\ge I(u,v)+ J(u,v)\,,
  		\end{equation}
  		where $I$ is as in \eqref{def:I} and $J\colon (L^1_\loc(\R^3))^2\to[0,+\infty]$ is defined as
  		\begin{equation}\label{def:J_2}
  			J(u,v)\defas \sigma \int_{J_u}\left(a_1\frac{1+v}{2}+a_2\frac{1-v}{2}\right)|H_{J_u}|^2\d\mathcal{H}^{2}\,,
  		\end{equation}
  		if $(u,v)\in BV(\R^3;\{-1,1\})\times BV(S;\{-1,1\})$, and  equal to $+\infty$ otherwise.
\end{proposition} 
%
In order to prove Proposition \ref{prop:lower-bound2}, we prove an intermediate result with an additional  dimension.
Given $(u_\eps)_{\eps>0}$ as above,  let 
\[
\tilde E_\eps^t=\{x\in\R^3:u_\e(x)>t\}
\] for every $\eps>0$ and $t\in\R$. Furthermore define 
$\tilde \mu_{\eps,t}\in \mathcal{M}(\R^3)$ by $\tilde \mu_{\eps,t}\defas\mathcal{H}^{2}\res\partial_*\tilde E_\eps^t$. Let  $\zeta_\eps\in \mathcal{M}(\R^{4})$ be defined by 
\begin{equation*}
        \int_{\R^{4}} g(x,t)\,{\rm d}\zeta_\eps(x,t)=\int_\R\int_{\R^3}g(x,t)\,{\rm d}\tilde \mu_{\eps,t}(x) \dt\,,
    \end{equation*}
    for every $g\in C^0_c(\R^3\times\R)$.

\begin{lemma}[Higher dimension]\label{ConvAddVar} Let $(u_\eps)_{\eps>0}$ be as in Proposition \ref{prop:lower-bound2}. Then there holds
    \begin{equation*}
         \lim_{\eps\to 0}  \int_{\R^{4}}|\eps|\nabla u_\eps(x)|-\sqrt{2\bar W(t)} |\,{\rm d}\zeta_\eps(x,t)=0\,.
    \end{equation*}
\end{lemma}

\begin{proof}
    By Proposition \ref{prop:discrepancy},
    \begin{equation}\label{VanDis}
        \lim_{\eps\to 0}\int_{\R^3} \bigg|\frac{\eps}{2}|\nabla u_\eps|^2-\frac{\bar W(u_\eps)}{\eps}\bigg| \dx=0\,.
    \end{equation}
    Multiplying the integrand with $\dfrac{\sqrt{\bar W(u_\eps)}}{\sqrt{\bar W(u_\eps)}+\frac{\eps}{\sqrt{2}}|\nabla u_\eps|}\leq 1$ leads to 
    \begin{equation}\label{eq1}
        \lim_{\eps\to 0}\int_{\R^3} \bigg|\frac{\eps}{2}|\nabla u_\eps|^2-\frac{\bar W(u_\eps)}{\eps}\bigg|\dfrac{\sqrt{\bar W(u_\eps)}}{\sqrt{\bar W(u_\eps)}+\frac{\eps}{\sqrt{2}}|\nabla u_\eps|} \dx=0.
    \end{equation}
    The  above integral can be rewritten as
    \begin{equation}\label{eq2}
        \int_{\R^3} \dfrac{\sqrt{\bar W(u_\eps)}}{\eps}\bigg|\frac{\eps}{\sqrt2}|\nabla u_\eps|-\sqrt{\bar W(u_\eps)} \bigg|\dx=\int_{\R^3} |\nabla u_\eps|\bigg| \frac{\sqrt{\bar W(u_\eps)}}{\sqrt{2}}-\frac{\bar W(u_\eps)}{\eps|\nabla u_\eps|}\bigg|\dx\,.
    \end{equation}
Now combining \eqref{VanDis}--\eqref{eq2} and using the triangle inequality we find
    \begin{equation*}
\begin{split}
     \lim_{\eps\to 0}  \int_{\R^3} & |\nabla u_\eps|\bigg| \frac{\eps}{2}|\nabla u_\eps|-\frac{\sqrt{\bar W(u_\eps)}}{\sqrt{2}} \bigg|\dx
  \\  & \le \lim_{\eps\to0}
     \int_{\R^3}  |\nabla u_\eps|\bigg| \frac{{\bar W(u_\eps)}}{\eps|\nabla u_\eps|}-\frac{\sqrt{\bar W(u_\eps)}}{\sqrt{2}} \bigg|\dx+ 
      \lim_{\eps\to 0}\int_{\R^3} \bigg|\frac{\eps}{2}|\nabla u_\eps|^2-\frac{\bar W(u_\eps)}{\eps}\bigg| \dx
     =0\,.
\end{split}
    \end{equation*}
Finally by the coarea formula
\begin{equation*}
\begin{split}
         \lim_{\eps\to 0}  \int_{\R^3}  |\nabla u_\eps|\bigg| \frac{\eps}{2}|\nabla u_\eps|-\frac{\sqrt{\bar W(u_\eps)}}{\sqrt{2}} \bigg|\dx&=  \lim_{\eps\to 0}  \int_\R\int_{\partial_* E_\eps^t}\bigg| \frac{\eps}{2}|\nabla u_\eps|-\frac{\sqrt{\bar W(t)}}{\sqrt{2}} \bigg|\dHn\dt\\
     & =\lim_{\eps\to 0}  \int_\R\int_{\R^3}\bigg| \frac{\eps}{2}|\nabla u_\eps|-\frac{\sqrt{\bar W(t)}}{\sqrt{2}} \bigg|\,{\rm d}\tilde \mu_{\eps,t}(x)\dt=0\,,
\end{split}
\end{equation*}
    and thus we conclude.
\end{proof}
We prove the following convergence of the pair $(u_\eps,v_\eps)_\e$:
\begin{lemma}[Strong convergence] 
\label{lem:strong_uv_convergence}
Let $(u_\eps,v_\eps)_{\eps}$ and $(u,v)$ be as in Proposition \ref{prop:lower-bound2}. Then 
    for every $\varphi\in C^0_c(\R^3\times\R)$ we have 
    \begin{equation*}
     \lim_{\eps\to 0}   \int_{\R^{4}}\eps|\nabla u_\eps(x)| \varphi(x,v_\eps(x))\,{\rm d}\zeta_\eps(x,t)=\sigma\int_{J_u}\varphi(x,v(x))\d \H^2(x).
    \end{equation*}
\end{lemma}

\begin{proof}
    By Lemma \ref{ConvAddVar}, it is sufficient to show that
    \begin{equation*}
     \lim_{\eps\to 0}   \int_{\R^{4}}\sqrt{2\bar W(t)} \varphi(x,v_\eps(x))\,{\rm d}\zeta_\eps(x,t)=\sigma\int_{J_u}\varphi(x,v(x))\d \H^2(x)\,,
    \end{equation*}
    for every $f\in C^0_c(\R^3\times\R)$.
     By definition of $\zeta_\eps$, and choosing $g(x,t)=\sqrt{2\bar W(t)}\varphi(x,v_\eps(x))$, we have
    \begin{equation*}
      \begin{split}
        \int_{\R^{4}}\sqrt{2\bar W(t)} \varphi(x,v_\eps)\,{\rm d}\zeta_\eps&=
        \int_\R\sqrt{2\bar W(t)}\int_{\R^3} \varphi(x,v_\eps)\,{\rm d}\tilde \mu_{\eps,t} \dt\\
&=\int_\R\sqrt{2\bar W(t)}\int_{\partial_* E_\eps^t} \varphi(x,v_\eps)\d \H^2 \dt
      \end{split}
    \end{equation*}
By  combining coarea formula and strict convergence, 
\[
  \begin{split}
    \int_{-\infty}^\infty\H^{2}(\partial_* \tilde E_\e^t)\d t&=|Du_\e|(\R^3)\to |Du|(\R^3)=\int_{-\infty}^\infty\H^{2}(\partial_* \tilde E^t)\d t\,,
  \end{split}
\]
with $\partial\tilde E^t=\{x\in \R^3: u(x)>t\}$. Hence,
\[
\H^{2}\res \partial_* \tilde E_\e^t\wsto
\begin{cases}
  \H^{2}\res J_u &\text{ for a.e. }t\in (-1,1)\\
0&\text{ for a.e. }t\in \R\setminus(-1,1)
\end{cases} \quad\text{ in } \mathcal M(\R^3)\,.
\]
Thus we obtain
\[
  \begin{split}
\lim_{\e\to 0}    \int_\R\sqrt{2\bar W(t)}\int_{\partial_* \tilde E_\eps^t} \varphi(x,v_\eps)\d \H^2 \dt&=\lim_{\e\to 0}\int_{-1}^1\sqrt{2\bar W(t)}\int_{\partial_* \tilde E_\eps^t} \varphi(x,v_\eps)\d \H^2 \dt\\
                             &=\sigma \int_{J_u} \varphi(x,v)\d \H^2\,,
  \end{split}
\]
with $v$ from Proposition \ref{prop:compactness2},
concluding the proof.
\end{proof}

\begin{lemma}
\label{lem:sqrt_strong_mfp}
With the above notation,   
\[
(\eps|\nabla u_\eps|^2 \L^3, \sqrt{a^{\bar \omega}(v_\eps)})\to (\sigma \H^2 \res \partial_* E, \sqrt{ a^{\bar\omega}(v)})
\]
in $L^q$ as measure-function pairs, for every $q\in[1,\infty)$.
\end{lemma}

\begin{proof}
Let $\varphi\in C^0_c(\R^{3}\times \R)$.  Using the coarea formula, we obtain 
\[
\begin{split}
\int_{\R^3}\varphi(x,\sqrt{a^{\bar \omega}(v_\eps(x))})\eps|\nabla u_\eps|^2\dx&=
\int_\R\int_{\partial_*\tilde E^t_\e}\varphi(x,\sqrt{a^{\bar \omega}(v_\eps(x))})\e|\nabla u_\e(x)|\d\H^{n-1}(x)\d t\\
&=\int_{\R^4}\varphi(x,\sqrt{a^{\bar \omega}(v_\eps(x))})\e|\nabla u_\e(x)|\d\zeta_\e(x,t)\,.
\end{split}
\]
Hence by 
Lemma \ref{lem:strong_uv_convergence}, 
\begin{equation*}
\lim_{\e\to 0}\int_{\R^3}\varphi(x,\sqrt{a^{\bar\omega}(v_\eps(x))})\eps|\nabla u_\eps|^2\dx=\sigma 
\int_{J_u} \varphi(x,\sqrt{a^{\bar \omega}( v(x))})\d \H^{2}\,.
\end{equation*}
It remains to show 
\[
\lim_{j\to\infty}\int_{\mathcal S_{\e,j}}|\sqrt{a^{\bar\omega}(v_\e)}|^q\e|\nabla u_\e|^2\d x\to 0\quad\text{ uniformly in  }\e>0\, ,
\]
where $\mathcal S_{\e,j}=\{x\in \R^3:|x|\geq j\text{ or } |\sqrt{a^{\bar\omega}(v_\e(x))}|\geq j\}$. 
  We note that the proof of this estimate is easier to achieve than in the  proof of Theorem \ref{thm1} in Step 4 thanks to the trivial $L^\infty$ bound $\sqrt{a^{\bar\omega}}\leq \max (\sqrt{a_1},\sqrt{a_2})$. 
Indeed, since $\e|\nabla u_\e|^2\L^3\leq \mu_\e\wsto \sigma \H^{2}\res J_u$, it follows from  Prokhorov's theorem 
\[
\lim_{j\to\infty}\int_{\{x:|x|\geq j\}}\e|\nabla u_\e|^2\d x= 0\quad\text{ uniformly in } \e>0\,.
\]
By the $L^\infty$-bound on $\sqrt{a^{\bar\omega}}$, we obtain 
\[
\lim_{j\to\infty}\int_{\{x:|x|\geq j\}}\e|\nabla u_\e|^2|a^{\bar\omega}(v_\e)|^{q/2}\d x= 0\quad\text{ uniformly in } \e>0\,.
\]
The remaining estimate 
\[
\lim_{j\to\infty}\int_{\{x:\sqrt{a^{\bar\omega}(v_\e(x))}\geq j\}}\e|\nabla u_\e|^2|a^{\bar \omega}(v_\e)|^{q/2}\d x=0 \quad\text{ uniformly in  } \e>0\,
\]
holds trivially by the boundedness of $\sqrt{a^\omega}$.
\end{proof}

We are now ready to prove Proposition \ref{prop:lower-bound2}.
\begin{proof}[Proof of Proposition \ref{prop:lower-bound2}] By Proposition \ref{prop:lower_bound} it readily follows that 
\begin{equation*}
    \liminf_{\eps\to0}I_\eps(u_\eps,v_\eps)\ge I(u,v)\,.
\end{equation*}
Thus it suffices to prove that 
\begin{equation*}
\liminf_{\eps\to0}J_\eps(u_\eps,v_\eps)\ge J(u,v)\,.
\end{equation*}
By Proposition \ref{prop:1stvariationVe} and Theorem \ref{thm:lowerbd_integrality} (i), there exists a subsequence (no relabeling) such that 
\[
\nabla u_\e\left(\frac{\bar W'(u_\e)}{\e}-\e\Delta u_\e\right)\L^3\wsto \sigma\H^{2}H_V\res J_u \quad\text{ in } \mathcal M(\R^3;\R^3)\,.
\]
Following \cite{roger2008allen}, we rewrite this as 
  \begin{equation}\label{eq:10}
    \e| \nabla u_\e|^2\frac{\e^{-1}\bar W'(u_\e)-\e\Delta u_\e}{\e|\nabla u_\e|}\frac{\nabla u_\e}{|\nabla u_\e|}\L^3\wsto \sigma\H^{2}H_V\res J_u \quad\text{ in } \mathcal M(\R^3;\R^3)\,.
\end{equation}
Additionally, 
  \begin{equation}\label{eq:9}
    \begin{split}
      \int_{\R^3} \e| \nabla u_\e|^2\left|\frac{\e^{-1}\bar W'(u_\e)-\e\Delta u_\e}{\e|\nabla u_\e|}\right|^2\d x
&= \int_{\R^3} \e^{-1}\left(\frac{\bar W'(u_\e)}{\e}-\e\Delta u_\e\right)^2\d x \\
      &\leq \frac{1}{\min( a_1,a_2) }J_\e(u_\e,v_\e)\\                                                        &\leq \Lambda\,.
    \end{split}
  \end{equation}
By \eqref{eq:10} and \eqref{eq:9},
\begin{equation}\label{eq:11}
        \bigg(\eps|\nabla u_\eps|^2\L^3, \dfrac{(\frac{\bar W'(u_\eps)}{\eps}-\eps \Delta u_\eps)\frac{\nabla u_\e}{|\nabla u_\e|}}{\eps|\nabla u_\eps|} \bigg)\rightharpoonup (\sigma \mathcal{H}^{2}\res J_u, H_{J_u})
    \end{equation}
    weakly as a measure-function pairs in $L^2$.
   By \cite[Proposition 3.2]{Mos01}, the weak  convergence from \eqref{eq:11} and the strong convergence from Lemma \ref{lem:sqrt_strong_mfp}  can be combined to obtain
    \begin{equation*}
        \bigg(\eps|\nabla u_\eps|^2 \L^3, \sqrt{a^\omega(v_\eps)}\dfrac{(\frac{\bar W'(u_\eps)}{\eps}-\eps \Delta u_\eps)}{\eps|\nabla u_\eps|}\frac{\nabla u_\e}{|\nabla u_\e|} \bigg) \rightharpoonup (\sigma \mathcal{H}^{2}\res J_u,\sqrt{a^{\bar\omega}(v)} H_{J_u})
    \end{equation*}
    weakly in $L^1$ (say) as measure-function pairs. Since $\sqrt{a^{\bar\omega}(v_\eps)}\dfrac{(\frac{\bar W'(u_\eps)}{\eps}-\eps \Delta u_\eps)}{\eps|\nabla u_\eps|}$ is uniformly bounded in $L^2_{\e|\nabla u_\e|^2\L^3}(\R^3)$, we have
    \begin{equation*}
        \bigg(\eps|\nabla u_\eps|^2 \L^3, \sqrt{a^{\bar\omega}(v_\eps)}\dfrac{(\frac{\bar W'(u_\eps)}{\eps}-\eps \Delta u_\eps)}{\eps|\nabla u_\eps|} \frac{\nabla u_\e}{|\nabla u_\e|}\bigg) \rightharpoonup (\sigma \mathcal{H}^{2}\res J_u,\sqrt{a^{\bar\omega}(v)} H_{J_u})
    \end{equation*}
    weakly in $L^2$ as measure-function pairs.
    The conclusion follows from the lower semi-continuity result for convex functionals with respect to weak measure-function pair convergence \cite[Theorem 4.4.2 (ii)]{Hut}.
\end{proof}


\appendix
\section{Upper bound in the smooth case}\label{AppendixA}
In this section we briefly discuss the construction of the recovery sequence when $(u,v)$ are such that $J_u$ and $J_v$ are smooth.  We give details for the upper bound construction only in the setting of Theorem \ref{thm2}, being the construction in the setting of Theorem \ref{thm1}, i.e., for $n\ge 2$ and any potential $W$ satisfying \eqref{hp:W}, exactly the same.
\begin{proposition}[Upper bound] Let $E,F$ be smooth subsets of $\R^3$ with $F\subset\partial E$. 
    Let $u=2\chi_{E}-1\in BV_{\rm loc.}(\R^3;\{-1,1\})$ and $v=2\chi_F-1\in BV(S;\{-1,1\})$ with $S=\llbracket \partial E,*\nu_u,1\rrbracket$. Then there exists a sequence $((u_\eps,v_\eps))$ with $u_\eps\in W^{2,2}_{\rm loc.}(\R^3)$, $v_\eps\in \hat H^{1,2}_{\mu_\eps}(\R^3)\cap L^\infty(\R^3)$ such that $u_\eps\to u$ strictly in $BV$,
 \begin{equation*}
    \left(|\nabla(\phi\circ u_\eps)|\L^n,v_\e\right)\to (\sigma \H^{n-1}\res J_u, v)\quad\text{ in }L^q
  \end{equation*}
for every $q\in [1,4)$, 
and
$$\limsup_{\eps\to0}I_\eps(u_\eps,v_\eps)\le I(u,v)\,, \quad \limsup_{\eps\to0}J_\eps(u_\eps,v_\eps)\le J(u,v)\,.$$
\end{proposition}
\begin{proof}
   In order to ensure $u_\eps\in W^{2,2}_{\rm loc.}(\R^3)$ we follow the argument of \cite{BP}.

Let  $(u,v)=(2\chi_{E}-1,2\chi_F-1)$ where $F,E$ are smooth subsets of $\R^3$ with  $F\subset \partial E$. Let ${\rm d}(x):={\rm dist}(x,\R^3\setminus E)-{\rm dist}(x,E)$ be the signed distance to $\partial E$. Letting ${\rm dist}_{\partial E}$ denote the geodesic distance on the smooth submanifold $\partial E$, we may define  the signed geodesic distance to $\partial F$ on $\partial E$ by  ${\rm d}_{\rm g}(y):={\rm dist}_{\partial E}(y,\partial E\setminus F)-{\rm dist}_{\partial E}(y,F)$ for  $y\in\partial E$.
Since $\partial E$ is smooth, for $\alpha>0$ sufficiently small the projection $\pi\colon \{x\in \R^3\colon |{\rm d}(x)|<\alpha\}\to\partial E$ is well defined. 
 For $x\in E_t\defas\{{\rm d}(x)=t\}$ we can write $\pi(x)= x-{\rm d}(x)\nu(x)$ with $\nu(x)={\rm sgn }({\rm d}(x))\frac{x-\pi(x)}{|x-\pi(x)|}$ the normal to $E_t$ at $x$. Moreover $\nu(x)$ coincides with the unit normal to $E$ at $\pi(x)$. We have that 
  \[
 \nabla\pi(x)=P_{T_xE^{{\rm d}(x)}}-{\rm d}(x)\nabla \nu(x)\,,
 \]
 where $P_{T_xE^{{\rm d}(x)}}$ denotes the projection onto the tangent space to $E^{{\rm d}(x)}$ at $x$.
  We recall that ${\rm d}$ and ${\rm d}_{\rm g}$ are $C^2$ in a sufficiently small tubular neighborhood of $\partial E$ and $\partial F$ respectively, $|\nabla {\rm d}(x)|, |\nabla^{\partial E}{\rm d}_{\rm g}(y)|=1$ (where $\nabla^{\partial E}$ denotes the tangential derivative) and 
  \begin{equation}\label{laplacian}
      -\Delta{\rm d}(x)=H^t(x)=\sum_{i=1}^{n-1}\frac{k_i(\pi(x))}{1-k_i(\pi(x)){\rm d}(x)}
\,,
  \end{equation}
where $H^t(x)$ denotes the sum of principal curvatures of the level set $E_t$ in $x\in E_t$ and $k_1(z),\dots, k_{n-1}(z)$ the principal curvatures of $\partial E$ at $z$.

We recall also that the function $w(t)=\tanh{(\sqrt{2}t)}$ is a solution to the minimization problem
\begin{equation*}
    \min\Big\{\int_{-\infty}^{+\infty}(\bar W(w)+\frac12(w')^2)\dt\colon w\in W^{1,2}_{\rm loc.}(\R), \ w(\pm\infty)=\pm1\Big\}\,.
\end{equation*}
Thus, in particular,
\begin{equation}\label{optimal-profile}
    \int_{-\infty}^{+\infty}(\bar W(w)+\frac12(w')^2)\dt= \int_{-\infty}^{+\infty}\sqrt{2\bar W(w)}|w'|\dt= \int_{-1}^{1}\sqrt{2\bar W(s)}\ds=\sigma\,,
\end{equation}
\begin{equation}\label{equipartition}
     \int_{-\infty}^{+\infty}|w'|^2\dt= \int_{-\infty}^{+\infty}\sqrt{2\bar W(w)}|w'|\dt=\sigma\,,
\end{equation}
and
\begin{equation}\label{Euler-Lagrange}
    w''(t)-\bar W'(w(t))=0\quad\forall t\in \R\,.
\end{equation}

 Set  $T_\eps\defas|\log\eps|$  and define $w_\eps\colon\R\to\R$ as
\begin{equation*}
    w_\eps(t)\defas\begin{cases}
        w(t)&\text{ if }t\in [0,T_\eps]\,,\\[4pt]
     p_\eps(t)&\text{ if }t\in (T_\eps,2T_\eps]\,,\\[4pt] 
     1&\text{ if }t\in (2T_\eps,+\infty)\,,\\[4pt] 
-w_\eps(-t)&\text{ if }t\in (-\infty,0)\,,
     \end{cases}
\end{equation*}
where $p_\eps\colon[T_\eps,2T_\eps]\to\R$ is a third degree polynomial chosen in such a way that $w_\eps\in C^{1,1}(\R)\cap C^{\infty}(\R\setminus\{\pm T_\eps,\pm2T_\eps\})$. Set also $\hat w_\eps(t)\defas w_\eps(t/\eps)$. One can verify that 
\begin{equation}\label{estimate-w-eps}
    \|\hat w_\eps'\|_{L^\infty(T_\eps,2T_\eps)}=o(\eps^2)\,,\quad 
    \|\hat w_\eps''\|_{L^\infty(T_\eps,2T_\eps)}=o(\eps^2)\,,\quad 
    \|w-\hat w_\eps\|_{L^\infty(T_\eps,2T_\eps)}=o(\eps)\,.
\end{equation}
We use $\hat w_\eps$, ${\rm d}$ and ${\rm d}_{\rm g}$ to construct  $u_\eps\colon\R^3\to\R$ and  $\tilde v_\eps\colon\partial E\to\R$, precisely, we set 
$$u_\eps(x)\defas \hat w_\eps({\rm d}(x))= w_\eps({\rm d}(x)/\eps)\quad \text{ and } \quad\tilde v_\eps(y)\defas \hat w_\eps({\rm d}_{\rm g}(y))= w_\eps({\rm d}_{\rm g}(y)/\eps)\,.$$

Then we let $v_\eps\colon \{x\in \R^3\colon {\rm dist}(x,\partial E)<\alpha\}\to\R$ be given by $v_\eps(x)\defas\tilde v_\eps(\pi(x))$ and take any extension in $\R^3$ such that $v_\eps\in W^{1,p}_{\rm loc.}(\R^3)$. 

We claim that the sequence $((u_\eps,v_\eps))$ satisfies the thesis.
By construction $u_\eps\to u$ strictly in $BV_{\rm loc.}(\R^3)$ and
 \begin{equation*}
    \left( |\nabla(\phi\circ u_\eps)|\L^n,v_\e\right)\to (\sigma \H^{n-1}\res J_u, v)\quad\text{ in }L^q
  \end{equation*}
as measure-function pairs for every $q\in [1,4)$. Moreover $(\tilde v_\eps)$ satisfies 
\begin{equation}\label{modica-mortola-recovery}
    \limsup_{\eps\to0}\int_{\partial E} \left(
    \frac\eps2|\nabla \tilde v_\eps|^2+\frac{\bar W(\tilde v_\eps)}{\eps}
    \right)\dHn(y)
    \le \sigma \mathcal{H}^1(\partial F)\,.
\end{equation}
For convenience we introduce the localized functional
\begin{equation*}
    I_\eps(u_\eps,v_\eps,A)\defas\int_A\left(
    \frac\eps2|\nabla u_\eps|^2+\frac{W(u_\eps)}{\eps}
    \right)
    \left(
    \frac\eps2|\nabla v_\eps|^2+\frac{W(v_\eps)}{\eps}
    \right)\dx\,,
\end{equation*}
with $A\subset \R^3$ open. Thus we have 
\begin{equation*}
    I_\eps(u_\eps,v_\eps)= I_\eps(u_\eps,v_\eps,\{|{\rm d}(x)|<\eps T_\eps\} )+ I_\eps(u_\eps,v_\eps,\{\eps T_\eps<|{\rm d}(x)|<2\eps T_\eps\} )\,.
\end{equation*}
Using that $|\nabla {\rm d}(x)|=1$, the coarea formula, and the change of variable $x=y+t\nu$ with $y=\pi(x)\in\partial E$ we have 
\begin{equation*}
    \begin{split}
        I_\eps(u_\eps,&v_\eps,\{|{\rm d}(x)|<\eps T_\eps\} )= \int_{\{|{\rm d}(x)|<\eps T_\eps\}}\left( \frac1{2\eps}\Big|w'\Big(\frac{{\rm d}(x)}{\eps}\Big)\Big|^2+ \frac1\eps \bar W\Big(w\Big(\frac{{\rm d}(x)}{\eps}\Big)\Big)\right)
        \left( \frac\eps2|\nabla v_\eps(x)|^2+ \frac{\bar W(v_\eps(x))}{\eps}
        \right)
        \dx\\
        &= \int_{-\eps T_\eps}^{\eps T_\eps}\int_{E_t} 
       \frac1\eps \left( \frac1{2}\Big| w'\Big( \frac t\eps\Big)\Big|^2+  \bar W\Big(w\Big( \frac t\eps\Big)\Big)\right)  \left( \frac\eps2|\nabla  \tilde v_\eps(\pi(x))|^2|\nabla \pi(x)|^2+ \frac{\bar W( \tilde v_\eps(\pi(x)))}{\eps}
        \right)\dHn(x)\dt \\
        & = (1+o(1))\int_{-\eps T_\eps}^{\eps T_\eps}
       \frac1\eps \left( \frac1{2}\Big| w'\Big(\frac t\eps\Big)\Big|^2+  \bar W\Big(w\Big(\frac t\eps\Big)\Big)\right) \dt\int_{\partial E} \left( \frac\eps2|\nabla \tilde v_\eps(y)|^2+ \frac{\bar W(\tilde v_\eps(y))}{\eps}
        \right)\dHn(y)\\
        &\le (1+o(1))\sigma \int_{\partial E} \left( \frac\eps2|\nabla \tilde v_\eps(y)|^2+ \frac{\bar W(\tilde v_\eps(y))}{\eps}
        \right)\dHn(y)\,.
    \end{split}
\end{equation*}
This together with \eqref{modica-mortola-recovery} yields 
\begin{equation*}
    \limsup_{\eps\to0}   I_\eps(u_\eps,v_\eps,\{|{\rm d}(x)|<\eps T_\eps\} )\le \sigma^2\mathcal{H}^2(\partial F)=\sigma^2\mathcal{H}^2(J_v)\,.
\end{equation*}
Similarly, from \eqref{estimate-w-eps}, one gets
\begin{equation}
    \begin{split}
        I_\eps(&u_\eps,v_\eps,\{\eps T_\eps<|{\rm d}(x)|<2\eps T_\eps\} )
       \\ & \le 2(1+o(1))\int_{\eps T_\eps}^{2\eps T_\eps}
        \left( \frac1{2\eps}|\hat w_\eps'(t)|^2+ \frac1\eps \bar W(\hat w_\eps(t))\right) \dt\int_{E_t} \left( \frac\eps2|\nabla \tilde v_\eps(y)|^2+ \frac{\bar W(\tilde v_\eps(y))}{\eps}
        \right)\dHn(y)\\
        &\le \frac C\eps \eps T_\eps\,,
    \end{split}
\end{equation}
from which
\begin{equation*}
    \limsup_{\eps\to0}   I_\eps(u_\eps,v_\eps,\{\eps T_\eps<|{\rm d}(x)|<2\eps T_\eps\} )=0\,.
\end{equation*}
Hence  we infer $\limsup_{\eps\to0}I_\eps(u_\eps,v_\eps)\le I(u,v)$. It remains to prove the upper bound for  $J_\eps(u_\eps,v_\eps)$.  
As before we write
\begin{equation*}
    J_\eps(u_\eps,v_\eps)= J_\eps(u_\eps,v_\eps,\{|{\rm d}(x)|<\eps T_\eps\} )+ J_\eps(u_\eps,v_\eps,\{\eps T_\eps<|{\rm d}(x)|<2\eps T_\eps\} )\,.
\end{equation*}
By \eqref{laplacian} and \eqref{Euler-Lagrange} it follows
\begin{equation*}
    \begin{split}
        \frac{\bar W'(u_\eps(x))}{\eps}-\eps\Delta u_\eps(x)&=\frac{\bar W'(\hat w_\eps({\rm d}(x)))}{\eps}- \eps \Big(\hat w_\eps''({\rm d}(x))+ \hat w_\eps'({\rm d}(x))\Delta{\rm d}(x) \Big)\\
        &=\frac{\bar W'(w_\eps({\rm d}(x)/\eps))}{\eps}-  \frac{w_\eps''({\rm d}(x)/\eps)}{\eps}+w_\eps'({\rm d}(x)/\eps)H^t(x)\\&= w_\eps'({\rm d}(x)/\eps)H^t(x)\,.
    \end{split}
\end{equation*}
From \eqref{laplacian} and the fact that $\|H\|_{L^\infty(\partial E)}<+\infty$ one can deduce that
\begin{equation*}
    H^t(x)=H(\pi(x))+o({\rm d}(x))=H(\pi(x))+o(\eps|\log\eps|)\quad\forall x\in \{|{\rm d}(x)|<\eps T_\eps\}\,.
\end{equation*}
Hence we get 
\begin{equation}\label{ub-J-1}
    \begin{split}
        J_\eps(u_\eps,&v_\eps,\{|{\rm d}(x)|<\eps T_\eps\} )=\int_{\{|{\rm d}(x)|<\eps T_\eps\} }\frac1\eps a^{\bar\omega}(v_\eps)\left(w_\eps'\Big(\frac{{\rm d}(x)}{\eps}
        \Big)\right)^2
        \left|H^t(x)\right|^2\dx\\
        & =
        \int_{\{|{\rm d}(x)|<\eps T_\eps\} }\frac1\eps a^{\bar\omega}(\tilde v_\eps(\pi(x)))\left(w'\Big(\frac{{\rm d}(x)}{\eps}
        \Big)\right)^2
        \Big(\left|H(\pi(x))\right|^2+ o(\eps^2|\log\eps|^2)\Big)\dx\\
        &= \frac1\eps\int_{-\eps T_\eps}^{\eps T_\eps} \left(w'\Big(\frac{t}{\eps}
        \Big)\right)^2\int_{E_t}
        a^{\bar\omega}(\tilde v_\eps(\pi(x)))
         \Big(\left|H(\pi(x))\right|^2+ o(\eps^2|\log\eps|^2)\Big)\dHn(x)\dt\\
        & = (1+o(1))\frac1\eps\int_{-\eps T_\eps}^{\eps T_\eps} \left(w'\Big(\frac{t}{\eps}
        \Big)\right)^2\dt\int_{\partial E}
        a^{\bar\omega}(\tilde v_\eps(y))
         \Big(\left|H(y)\right|^2+ o(\eps^2|\log\eps|^2)\Big)\dHn(y)
        \,.
    \end{split}
\end{equation}
By \eqref{equipartition} we find
\begin{equation}\label{ub-J-2}
    \frac1\eps\int_{-\eps T_\eps}^{\eps T_\eps} \left(w'\Big(\frac{t}{\eps}
        \Big)\right)^2\dt =\int_{-T_\eps}^{T_\eps}|w'|^2\dt\le \sigma\,.
\end{equation}
Whereas 
\begin{equation}\label{ub-J-3}
    \limsup_{\eps\to0} \int_{\partial E}
        a^{\bar\omega}(\tilde v_\eps(y))
         \Big(\left|H(y)\right|^2+ o(\eps^2|\log\eps|^2)\Big)\dHn(y)\le 
         \int_{\partial E}
        a^{\bar\omega}(v(y))
\left|H(y)\right|^2\dHn(y)=J(u,v)\,.
\end{equation}
Combining \eqref{ub-J-1}--\eqref{ub-J-3} we get 
\begin{equation*}
\limsup_{\eps\to0}J_\eps(u_\eps,v_\eps,\{|{\rm d}(x)|<\eps T_\eps\} )\le J(u,v)
\end{equation*}
Finally from the second equality in \eqref{estimate-w-eps} one can easily deduce that
\begin{equation*}
\limsup_{\eps\to0}J_\eps(u_\eps,v_\eps,\{\eps T_\eps<|{\rm d}(x)|<2\eps T_\eps\} )=0\,,
\end{equation*}
and the proof is concluded.
\end{proof}

\section{Uniqueness for the gradient in Sobolev spaces with respect to measures}\label{AppendixB}

As in Section 2.3, let $M$ be a $k$-rectifiable, $\tau:M\to\Lambda^k(\R^n)$ an orientation of $M$, $\rho:M\to[0,\infty)$ locally $\H^k\res M$ integrable, and $S=\llbracket M,\tau,\rho\rrbracket$.
In the proof of the uniqueness of the gradient for $u\in H^{1,p}(S)$, we will use the following lemma:

\begin{lemma}{\cite[Corollary 2.10(ii)]{ADS}}
 Let $S=\llbracket M,\tau,\rho\rrbracket$ and $u\in BV(S)$. Then there exists a $\Lambda_{n-k+1}(\R^n)$ valued measure $R(S,u)$ whose components $R^\gamma(S,u)$, $\gamma\in \Lambda(n,n-k+1)$, are given by 
\[
\langle \varphi,R^\gamma(S,u)\rangle=\int_M u (\nu\wedge\nabla \varphi )^\gamma \rho\d\H^k \quad\text{ for all } \varphi\in C^1_c(\R^n)\,.
\] 
\end{lemma}

In the upcoming lemma, we write $\mu:=\|S\|$. We will assume that $\partial S=0$, since this is the case that we are working with, and this alleviates the calculations.

\begin{lemma}
\label{lem:gradient_uniqueness_current}
Assume that $S$ is as above with $\partial S=0$. 
Let $(v_j)_{j\in \N}\subset C^\infty_c(\R^n)$ such that
    \begin{equation*}
        \lim_{j\to\infty}\|v_j\|_{L^p_{\mu}(\R^n)}= 0 \quad\textrm{ and }\quad \sup_{j\in\N}\|\nabla v_j\|_{L^p_{\mu}(\R^n)}<C\,,
\end{equation*}
     then $\nabla v_j\wto 0  $ in $L^p_{\mu}(\R^n)$.
\end{lemma}

\begin{proof}
  Let $\varphi\in C_c^\infty(\R^n)$. Then $R(S,\varphi)$ is in $C^\infty_c(\R^n;\Lambda_{n-k+1}(\R^n))$, and for $\gamma\in\Lambda(n,n-k+1)$ we have that
\[
0=\lim_{j\to\infty} \int_M R^\gamma(S,\varphi)v_j\rho\d\H^k=-\lim_{j\to\infty}\int_M R^\gamma(S,v_j)\varphi\rho\d\H^k\, 
\]
Since $R(S,v_j)=\nu\wedge\nabla_\mu v_j$ and by the boundedness of $(\nabla_\mu v_j)_{j\in\N}$  in $L^p_{\mu}(\R^n;\R^n)$, there exists a subsequence (no relabeling) and $w\in L^p_{\mu}(\R^n;\R^n)$ such that $\nabla_\mu v_j\wto w$ in $L^p_\mu(\R^n;\R^n)$, and 
\[
\begin{split}
\lim_{j\to\infty}\int_M R^\gamma(S,v_j)\varphi\rho\d\H^k&=\int (\nu\wedge\nabla_\mu v_j)^\gamma\varphi\rho\d\H^k\\
&=\int (\nu\wedge w)^\gamma\varphi\rho\d\H^k\,.
\end{split}
\] 
Since $\nabla_\mu v_j$ is orthogonal to $\nu$, the same holds for $w$, and hence $w=0$ follows since $\varphi$ can be chosen arbitrarily.
\end{proof}

\begin{lemma}\label{lem:gradient_uniqueness_bulk}
Let $\e>0$,  $W:\R\to[0,\infty)$ as in Theorem \ref{thm1},  $\mu_\e=\left(\frac{\e}{2}|\nabla u_\e|^2+\e^{-1}W(u_\e)\right)\L^n$.    If $(v_k)_{k\in\N}\subset C^\infty_c(\R^n)$ is such that
    \begin{equation*}
        \lim_{k\to\infty}\|v_k\|_{L^p_{\mu_\eps}(\R^n)}= 0 \quad\textrm{ and }\quad \sup_{k\in\N}\|\nabla v_k\|_{L^p_{\mu_\eps}(\R^n)}<C\,,
    \end{equation*}
   then $\nabla v_k\wto 0  $ in $L^p_{\mu_\e}(\R^n)$.

\end{lemma}
\begin{proof}
It suffices to show that for any subsequence, there exists a further subsequence for which the gradient converges weakly to 0 in $L^p_{\mu_\e}$. So let us start with an arbitrary subsequence. By boundedness of the sequence in $L^p_{\mu_\e}$, we may choose a further subsequence that is weakly convergent to some limit $f$,
\[
\nabla v_k\wto f\quad\text{ in }L^p_{\mu_\e}\,.
\]
Here and in the following, we do not relabel when we take subsequences. 

\medskip

Suppose that $ -1<t_1<t_2<1$. Then 
\begin{equation*}
\begin{split}
    \int_{t_1}^{t_2}\int_{\partial_*\{u_\eps\ge t\}}|v_k|\dHn \dt&= \int_{\{t_1\leq u_\eps \leq t_2\}}| v_k\|\nabla u_\eps|\dx\\
    &\le \dfrac{\eps \|\nabla u_\eps\|_{L^2(\R^n)}}{C(W,t_1,t_2)}\bigg(\int_{\{t_1\leq u_\eps \leq t_2\}} \frac{W(u_\eps)}{\eps}|v_k|^2 \dx \bigg)^2\\
    &\le \dfrac{\eps \|\nabla u_\eps\|_{L^2(\R^n)}}{C(W,t_1,t_2)} \|v_k\|^2_{L^2_{\mu_\eps}(\R^n)}\to 0\quad \text{as }k\to+\infty\,.
\end{split}
\end{equation*}
Therefore, there exists a subsequence  such that for  almost every $s_1,s_2$ with $t_1<s_1<s_2<t_2$, we have that $\{s_1\leq u_\eps \leq s_2\}$ is a  set of finite perimeter and 
\begin{equation}
\label{eq:Hn-1bound}
    \int_{\partial_*\{u_\eps\ge s_i\}}|v_k|\dHn \to 0\quad \text{as }k\to+\infty\quad \text{ for }i=1,2\,.
\end{equation}
For every $\varphi\in C^\infty_0(\R^n)$ we can apply the Gauss-Green theorem:
 \begin{equation*}
    \int_{\{s_1\leq u_\eps \leq s_2\}}\nabla v_k\cdot \varphi \dx= \int_{\{s_1\leq u_\eps \leq s_2\}}v_k\textrm{ div }\varphi\dx+\int_{\partial_*\{s_1\leq u_\eps \leq s_2\}}v_k\langle \varphi,\nu_{u_\eps}\rangle\dHn\,.
 \end{equation*}
 The last term converges to $0$ by choice of $s_1$ and $s_2$.
Moreover, we have 
    \begin{equation*}
         \begin{split}\int_{\{s_1\leq u_\eps \leq s_2\}} |v_k|^2\dx&\leq
         \frac{\eps}{C(W,s_1,s_2)}\int_{\{s_1\leq u_\eps 
         \leq s_2\}} \frac{W(u_\eps)}{\eps}|v_k|^2\dx\\
         &\leq 
            \frac{\eps}{C(W,s_1,s_2)}
          \|v_k\|^2_{L^2_{\mu_\eps}(\R^n)}\\
&         \to 0\quad \text{as }k\to+\infty\,.
         \end{split}
    \end{equation*}
   Thus,
   \begin{equation*}
       \int_{\{s_1\leq u_\eps \leq s_2\}}\nabla v_k\cdot \varphi\dx\to 0\quad \text{as }k\to+\infty\,.
   \end{equation*}
This implies that $f=0$ on $\{x:s_1\leq u_\eps(x) \leq s_2\}$. Since we may choose $t_1,s_1$ arbitrarily close to $-1$, and $t_2,s_2$ arbitrarily close to $1$, we obtain $f=0$ on $\{x:-1<u_\e(x)<1\}$.  
 
 \medskip
 
With the same arguments we can prove that $f=0$ on $\{x:u_\e(x)<-1\}$ 
and on  $\{x:u_\e(x)>1\}$. Since $\mu_\e(\{x:u_\e(x)=\pm1\})=0$, we obtain that $f=0$ $\mu_\e$-almost everywhere, which completes the proof of the present lemma.
\end{proof}


\begin{thebibliography}{}
\bibitem[AFP00]{AmFuPa} {\sc L.~Ambrosio, N.~Fusco and D.~Pallara},  Functions of bounded variation and free discontinuity
              problems. Oxford Mathematical Monographs, The Clarendon Press, Oxford University Press, New York (2000).
 \bibitem[ABCP03]{AnBraCP} 
 	{\sc N.~Ansini, A.~Braides, and V.~Chiad\'{o} Piat}, 
 	\newblock Gradient theory of phase transitions in composite media,
 	\newblock {\em Proc. R. Soc. Edind.} {\bf 133A} (2003), 265--296.             
\bibitem[ADS96]{ADS} {\sc G.~Anzellotti, S.~Delladio, and G.~Scianna}, \newblock{B{V} functions over rectifiable currents, }\newblock{\em Ann. Mat. Pura Appl. }(4), {\bf 170}  p257-296, (1996).
\bibitem[BEMZ22]{BEMZ1}
 {\sc A.~Bach, T.~Esposito, R.~Marziani and C.~I.~Zeppieri},
\newblock{Interaction between oscillations and singular perturbation in a one-dimensional phase-field model},
\newblock{\em AWM Volume - Research in the Mathematics of Materials Science. Association for Women in Mathematics Series, vol 31. Spronger, Cham.} (2022). 
\bibitem[BEMZ23]{BEMZ2}
 {\sc A.~Bach, T.~Esposito, R.~Marziani and C.~I.~Zeppieri},
 \newblock{Gradient damage models for heterogeneous materials.}
 \newblock{\em Siam J. Math. Anal.} \textbf{55}, 4 (2023).
 \bibitem[BMZ23]{BMZ}
 {\sc A.~Bach, R.~Marziani and C.~I.~Zeppieri},
 \newblock{$\Gamma$-convergence and stochastic homogenisation of singularly perturbed elliptic functionals},
 \newblock{\em Calc. Var.} \textbf{62}, 199 (2023). 
 \bibitem[Bal90]{Baldo} 
	{\sc S.~Baldo}, 
	\newblock Minimal interface criterion for phase transitions in mixtures of Cahn-Hilliard fluids, 
	\newblock{\em Ann. Inst. H. Poincaré Anal. Non Linéaire}, {\bf 7} (1990), 67--90.
 \bibitem[BF94]{BaFo} 
	{\sc A.~C.~Barroso and I.~Fonseca}, 
	\newblock Anisotropic singular perturbations - the vectorial case, 
	\newblock{\em Proc. Roy. Soc. Edinburgh Sect. A} {\bf 124} (1994), 527--571.
 \bibitem[BP93]{BP}  {\sc G.~Bellettini,  and  M.~Paolini},  \newblock{Approssimazione variazionale di funzionali con curvatura}, \newblock{\em Seminario di Analisi Matematica, Dipartimento di Matematica dell'Universit\'a di Bologna, 11 marzo} (1993), 87--97.
 \bibitem[Bog07]{bogachev2007measure}
{\sc V.~Bogachev}
\newblock {\em Measure Theory, Vol. 1 and 2}.
\newblock {\em Springer}, 2007.
 \bibitem[Bou90]{Bouchitte} 
	{\sc G.~Bouchitt\'e},
	\newblock Singular perturbations of variational problems arising from a two-phase transition model,
	\newblock{\em Appl. Math. and Opt.} {\bf 21} (1990), 289--314.	 

\bibitem[BBS97]{BBS}
 {\sc G.~Bouchitt\'e, G.~Buttazzo, and P.~Seppecher}
\newblock Energies with respect to a measure and applications
to low dimensional structures.
\newblock {\em Calculus of Variations and PDE}, 5:37-54, (1997).

\bibitem[BBF01]{BBF}
G.~Bouchitt\'e, G.~Buttazzo, and I.~Fragal\`a. 
\newblock Convergence of Sobolev spaces on varying manifolds.
\newblock {\em The Journal of Geometric Analysis} 11 (2001): 399-422
 	\bibitem[BZ09]{BraZep} 
	{\sc A.~Braides and C.~I.~Zeppieri}, 
	\newblock Multiscale analysis of a prototypical model for the interaction between microstructure and surface energy, \newblock{\em Interfaces Free Bound.} {\bf 11} (2009), 61--118.
 \bibitem[BLS20]{BrLS20}
{\sc K.~Brazda, L.~Lussardi, and U.~Stefanelli}
\newblock Existence of varifold minimizers for the multiphase
  {Canham–Helfrich} functional,
\newblock {\em Calc. Var. Partial Dif.}, 59(3):26, May 2020.
\newblock Id/No 93.
\bibitem[CMV13]{ChMV13}
{\sc R.~Choksi, M.~Morandotti, and M.~Veneroni}
\newblock Global minimizers for axisymmetric multiphase membranes.
\newblock {\em ESAIM: Control, Optimisation and Calculus of Variations},
  19(4):1014--1029, July 2013.
 \bibitem[CFG23]{CFG}
	{\sc R.~Cristoferi, I.~Fonseca, and  L.~Ganedi}
	\newblock{Homogenization and phase separation with space dependent wells -the subcritical case.},
	\newblock{\em  Arch. Rational Mech. Anal.}  \textbf{247}, 94 (2023).
	\bibitem[CFHP19]{CFHP}
	{\sc R.~Cristoferi, I.~Fonseca, A.~Hagerty, and C.~Popovici}
	\newblock{A homogenization result in the gradient theory of phase transitions},
	\newblock{\em  Interfaces and Free Boundaries} \textbf{21} (2019), 367–408.
 \bibitem[EG92]{evans1992measure}
{\sc L.~C.~Evans, and R.~F.~Gariepy}
\newblock {\em Measure Theory and Fine Properties of Functions}.
\newblock{\em CRC Press, 1992}.
\bibitem[FM99]{FM}
{\sc I.~Fragal\`a, and C.~ Mantegazza}
\newblock On some notions of tangent space to a measure.
\newblock {\em Proceedings of the Royal Society of Edinburgh Section A: Mathematics} 129.2 (1999): 331-342.
 \bibitem[Hel12]{Helm13}
{\sc M.~Helmers}
\newblock Kinks in two-phase lipid bilayer membranes.
\newblock {\em Calc. Var. Partial Dif.}, 48(1-2):211--242, August 2012.
\bibitem[Hel14]{Helm15}
{\sc M.~Helmers}
\newblock Convergence of an approximation for rotationally symmetric two-phase
  lipid bilayer membranes.
\newblock {\em Q. J. Math.}, 66(1):143--170, October 2014.
\bibitem[Hut86]{Hut} {\sc John~E.~Hutchinson}, \newblock{Second fundamental form for varifolds and the existence of surfaces minimising curvature}, \newblock{\em Indiana Univ. Math. J.} {\bf 35} (1986), p45-71.
\bibitem[JL96]{JuLi96}
{\sc F.~Jülicher, and R.~Lipowsky}
\newblock Shape transformations of vesicles with intramembrane domains,
\newblock {\em Phys. Rev. E}, 53(3):2670--2683, March 1996.
\bibitem[KS04]{KuSc04}
{\sc E.~Kuwert, and R.~Schätzle}
\newblock Removability of point singularities of willmore surfaces,
\newblock {\em Ann. Math.}, 160(1):315--357, July 2004.

\bibitem[LY82]{li1982new}
{\sc P.~Li, and S.~T.~Yau}
\newblock A new conformal invariant and its applications to the {W}illmore
  conjecture and the first eigenvalue of compact surfaces,
\newblock {\em Inventiones mathematicae}, 69(2):269--291, 1982.	
\bibitem[MSZ03]{MSZ} {\sc J.~Mal\'y, D.~Swanson, and W.P.~Ziemer} \newblock{The co-area formula for {S}obolev mappings},
\newblock{\em Trans. Amer. Math. Soc.} \textbf{ 355} (2003), 477-492.
 \bibitem[Mar23]{Marziani} R.~Marziani, 
 \newblock{$\Gamma$-convergence and stochastic homogenisation of phase-transition functionals},
 \newblock{\em  ESAIM: Control Optim. Calc. Var.} \textbf{29} (2023) 44.
\bibitem[MM77]{MoMo77}
{\sc L.~Modica, and S.~Mortola}
\newblock Un esempio di {$\Gamma$}-convergenza.
\newblock {\em Boll. Un. Mat. Ital. B (5)}, 14(1):285--299, 1977.
\bibitem[Mod87]{Modi87}
{\sc L.~Modica}
\newblock The gradient theory of phase transitions and the minimal interface
  criterion,
\newblock {\em Arch. Ration. Mech. An.}, 98(2):123--142, June 1987.
\bibitem[Mor20]{Morfe}
	{\sc P.~S.~Morfe},
	\newblock Surface Tension and $\Gamma$-Convergence of Van der Waals-Cahn-Hilliard Phase Transitions in Stationary Ergodic Media,   
	\newblock {\em J. Stat. Phys.} {\bf 181} (2020), no. 6, 2225--2256. 
\bibitem[Mos01]{Mos01} {\sc R.~Moser} \newblock{A generalization of Rellich’s theorem and regularity of varifolds minimizing curvature},
        \newblock{\em Preprint 72, MPI MiS Leipzig}, (2001).
\bibitem[OR23]{OR} {\sc H.~Olbermann, and M.~R\"{o}ger}, \newblock{Phase separation on varying surfaces and convergence of diffuse interface approximations}, \newblock{\em Calc. Var. Partial Differential Equations} {\bf 62} (2023).
  \bibitem[OS91]{OwSte}
	{\sc N.~C.~Owen and P.~Sternberg}, 
	\newblock Nonconvex variational problems with anisotropic perturbations, 	
	\newblock {\em Nonlinear Anal.} {\bf 16} (1991), 705--719.   
 \bibitem[RM08]{roger2008allen}
{\sc M.~R{\"o}ger,   and L.~Mugnai}
\newblock The Allen--Cahn action functional in higher dimensions,
\newblock {\em Interfaces and Free Boundaries},  10(1):45--78, 2008.
	\bibitem[RS06]{RS} {\sc M.~R\"{o}ger, and R.~Sch\"{a}tzle},  \newblock{On a modified conjecture of {D}e {G}iorgi}, \newblock{\em Mathematische Zeitschrift} {\bf 254}, p675--714 (2006).
 \bibitem[Sim83]{Simo83}
{\sc L.~Simon},
\newblock {\em Lectures on geometric measure theory}, volume~3 of {\em
  Proceedings of the Centre for Mathematical Analysis, Australian National
  University}.
\newblock Australian National University, Centre for Mathematical Analysis,
  Canberra, 1983.
\bibitem[Ton02]{tonegawa2002phase}
{\sc Y.~Tonegawa}
\newblock Phase field model with a variable chemical potential,
\newblock {\em Proceedings of the Royal Society of Edinburgh Section A: Mathematics} 132.4 (2002): 993-1019.




\end{thebibliography}
\end{document}